\documentclass [a4paper, 10pt] {amsart} 
\usepackage {amssymb,amsmath,bbm,latexsym,setspace,parskip}
\input diagrams.sty

\newtheorem{thm}{Theorem}
\newtheorem{prop}[thm]{Proposition}
\newtheorem{lemma}[thm]{Lemma}
\newtheorem{cor}[thm]{Corollary}
\newtheorem{dfn}[thm]{Definition}

\newtheorem{rmk}[thm]{Remark}
\newtheorem{ex}{Question}

\numberwithin{thm}{section}

\newcommand {\essup} {\mathop \mathrm{ess \; sup}}

\newcommand {\Belt} {\mathrm{Belt}}
\newcommand {\out} [1] {{}}

\newcommand {\shb} {\stackrel{!}{=}}

\newcommand {\bb} [1] {\mathbbm{#1}}
\newcommand {\cal} [1] {\mathcal{#1}}
\newcommand {\ol} [1] {\overline{#1}}

\newcommand {\cinf} {\hat{\bb C}}

\newcommand {\MPIs} {Max Planck Institute for Mathematics in the Sciences}
\newcommand {\be}{\begin{equation}\begin{aligned}}
\newcommand {\ee}{\end{aligned}\end{equation}}
\newcommand {\benn}{\begin{equation*}\begin{aligned}}
\newcommand {\eenn}{\end{aligned}\end{equation*}}

\title{Higher Bers Maps}
\author{Guy Buss} 
\date{\today}

\begin{document}
\begin{abstract}
   The Bers embebbing realizes the Teichm\"uller space of a Fuchsian group $G$ as a open, bounded and contractible subset of the complex Banach space of bounded quadratic differentials for $G$. It utilizes the schlicht model of Teichm\"uller space, where each point is represented by an injective holomorphic function on the disc, and the map is constructed via the Schwarzian differential operator.\\ \\
   In this paper we prove that a certain class of differential operators acting on functions of the disc induce holomorphic mappings of Teichm\"uller spaces, and we also obtain a general formula for the differential of the induced mappings at the origin. The main focus of this work, however, is on two particular series of such mappings, dubbed higher Bers maps, because they are induced by so-called higher Schwarzians -- generalizations of the classical Schwarzian operator. For these maps, we prove several further results.\\ \\
   The last section contains a discussion of possible applications, open questions and speculations. 
\end{abstract}
\maketitle
\tableofcontents
\section{Introduction}
Teichm\"uller spaces are central objects in geometry today, with deep connections to various other seemingly unrelated topics. There are several ways to represent them, of which three will come up in this paper. One of the most useful representations of Teichm\"uller spaces is via the \emph{Bers embedding}, which realizes the Teichm\"uller space of a Fuchsian group $G$ as an open contractibe domain in the complex Banach space $B_2(\bb D, G)$ of bounded quadratic differentials for $G$. The Bers embedding relies on a different representation of Teichm\"uller space, the so-called schlicht model (see Section \ref{MoTS}), where each point is represented by a schlicht (or univalent) function on the disc, and is constructed with the help of the classical Schwarzian differential operator. \\ \\
In the present paper we investigate as to whether other differential operators yield holomorphic maps of Teichm\"uller spaces into other complex Banach spaces. The main result on which the rest of the work builds upon is the following theorem (for a more precise and technical version, see Thm.~ \ref{higherBersholomorphic}). 
\begin{thm}
	  Let $Q$ be a differential operator that maps schlicht functions to holomorphic functions and that satisfies 
	  $$ Q[f\circ g] = (Q[f]\circ g)(g')^{m} \qquad \forall g \in \mathrm{PSL}(2, \bb C)\;,$$
	  for some $m \in \bb N$. Further, suppose $Q[f]$ is a polynomial in $f'',\ldots, f^{(N)}$ and $(f')^{-1}$ with           complex coefficients. Then $Q$ induces a holomorphic map of any Teichm\"uller space into the complex Banach space  			of $m$-differentials.
\end{thm}
All necessary concepts and notations are explained in Section \ref{Bersemb}. This section also contains a considerable amount of background information concerning the Bers embedding, so that the results of this paper can be viewed with the right perspective. \\ \\
Section \ref{HMoTS} contains the aforementioned theorem and its proof, as well as Theorem \ref{ThmDifferential}, describing the derivatives at the origin of the maps $\beta^Q$, for any admissibe $Q$ (i.e., $Q$ satisfies the assumptions of the above theorem), as bounded operators between Banach spaces.\\ \\
In Section \ref{HBM}, we show that there do exist admissible operators. We focus on two series of maps, $\beta^A_n$ and $\beta^B_n$, $n\geq 3$, which we call \emph{higher Bers maps} because they are constructed with the help of the $A$ and $B$ series of \emph{higher Schwarzians} -- generalizations of the Schwarzian derivative that are discussed in Section \ref{HSD}.\\ \\
For these higher Bers maps, we prove several results. First of all, we establish the surjectivity of the differential at the origin (Thm.~ \ref{HBDiffSurj}). Then we study the kernel, for which we derive a precise expression in Thm.~ \ref{KernelHBM}. This description implies, in particular, that the differential is injective when restricted to the tangent space of all Teichm\"uller spaces except the universal one (Cor. \ref{InjOnTan}).\\ \\
Subsection \ref{HBM-SGR} contains two results (Thm.~ \ref{HighInj} and Thm.~ \ref{HighInj2}) that are not infinitesimal in nature and state that the fibre of the higher Bers maps at the origin consists of a single point.\\ \\
Of course, the main question is whether the higher Bers maps are embeddings, and the results of this paper hint that this might be the case. Yet there are still several technical difficulties to overcome. This, as well as several other interesting questions and other possible applications of the higher Bers maps, are discussed in the last section. 
\subsection{Acknowledgements}
This work is part of my Ph.D. thesis written under the supervision of J\"urgen Jost at the \MPIs{} in Leipzig, whom I thank sincerely for his constant support and encouragement. I also want to thank Brian Clarke for careful proofreading and both him and Christoph Sachse for many inspiring discussions. The work was funded both by the Research Training Group \emph{Analysis, Geometry and their Interaction with the Natural Sciences} at the University of Leipzig and the \emph{International Max Planck Research School} at the MPI Leipzig.
\section{Models of Teichm\"uller Spaces and the Bers Embedding} \label{Bersemb}
The purpose of this section is to set up terminology and introduce some notation, as well as to cast the later work in the right perspective. The material in this section is completely standard and can be found in any reference on Teichm\"uller theory, e.g., \cite{TCAToTS}.\\ \\
For the rest of  the paper, let $\cinf$ denote the Riemann sphere, $\bb D$ the unit disc and $\bb D^c$ the interior of its complement in $\cinf$. If $A$ is a Banach space, let $\bb B_r(A,p)$ denote the open ball of radius $r$ around $p \in A$. If $p$ is not specified, it is assumed to be the origin. For a function defined on a domain in $D \subset \bb C$ we write $\partial f$ or $f'$ for its holomorphic derivative and $\bar \partial f$ for the anti-holomorphic derivative.
\subsection{Models of Teichm\"uller Space} \label{MoTS}
Let $G$ be an arbitrary Fuchsian group acting on the unit disc $\bb D$. The Teichm\"uller space of $G$ can be realized in several ways, one of which being the \emph{Beltrami model} $\cal T_B(G)$. To this end, let
$L^\infty_{(-1,1)}(\bb D, G)$ denote the Banach space of measurable complex valued functions on $D$ that satisfy
$$ \mu(z) = \mu (g z)\cdot
\bar \partial g(z)/\partial g(z)\;, \qquad \mathrm{ for
\;almost\; all} \: z \in D, \; \forall \;g \in
G\;.$$
The Banach norm used here is the usual sup-norm. Points of the Beltrami model are given by equivalence classes of these functions of norm less than one, denoted by $\Belt( \bb D^c,G)$,
$$ \cal T_B(G) \cong \bb B_1\left(L^\infty_{(-1,1)}(\bb D^c, G)\right)/{\sim} \;=\;\Belt(\bb D^c, G)/{\sim}\;.$$
The equivalence relation utilizes the solution of the Beltrami equation. More precisely, given $\mu \in \bb B_1(L^\infty(\bb C))$. Then there exists a solution $w[\mu]$ to the Beltrami equation,
$$ \bar \partial w[\mu] (z) = \mu (z) \partial w[\mu] (z) \qquad \mathrm{ for
\;almost\; all} \: z \in \bb C\;,$$
which is necessarily quasiconformal, and this solution is unique up to post-composition by M\"obius transformations (see, e.g., \cite{RMTfvM}, \cite{LoQM} or \cite{QA}).\\ \\
It is also crucial for Teichm\"uller theory that if we have a family of Beltrami differentials depending on some parameters $t_i$, the reqularity of the solutions in $t_i$ is `at least as good' as the regularity of the family itself (for a more precise statement, see \cite{LoQM}). In particular, if we have a family of Beltrami differentials depending holomorphically on a paramater $t$, then the properly normalized solutions are holomorphic in $t$. For small $t$ we can hence write
$$ w^{t\nu}(z) = w^\nu_0(z) + tw^\nu_1(z) + t^2 w_2^{\nu}(z) + \ldots\;.$$
There is a closed expression for the first order approximation $w^\nu_1$ (see again, e.g., \cite{LoQM}), which we will use in the proof of Thm.~ \ref{ThmDifferential},
\begin{eqnarray} \label{1stOrderBeltSol}
   w_1^{\nu} (z) = -\frac{z(z-1)}{\pi}\int_{\bb C} \frac{\nu(w)}{w(w-1)(w-z)} d^2w\;.
\end{eqnarray}
Now, in order to use the theorem on the existence of a solution to define an equivalence relation, one has to choose an extension of the elements in $L^\infty_{(-1,1)}(\bb D^c, G) \subset L^\infty(\bb D^c)$ to elements of $L^\infty(\bb C)$. There are two canonical choices of which we use the one where the function is set equal to zero outside of $\bb D^c$. We denote the solution to the Beltrami equation by $w^\mu$. Observe that $w^\mu$ is an injective holomorphic injective function on the unit disc, and thus we can get rid of the M\"obius degree of freedom by \emph{1-point normalizing} $w^\mu$, i.e. we require $w^\mu (0) = 0, (w^\mu)'(0)=1$ and $(w^\mu)''(0)=0$. This is implicitly assumed in the symbol $w^\mu$. One can see that the family $\cal F_\delta$ of all quasiconformal homeomorphisms normalized in this way and with dilatation of norm bounded by $\delta < 1$ is a \emph{normal family} as in the case of the more common 3-point normalization.\\ \\
Now we are in position to describe the equivalence relation: $\mu, \nu \in \bb B_1(L^\infty(\bb D^c))$ are called equivalent iff $w^\mu_{|\partial \bb D} = w^\nu_{|\partial \bb D}$. \\ \\
The second model we will utilize is the \emph{schlicht model} $\cal T_S(G)$, which is defined by
$$ 
	\cal T_S(G):=\left\{ 
         \begin{array}{l}
            f \in \cal S^0(\bb D):\; \exists \: \textrm{an or.-pres. quasiconf. homeo.}\\
            w: \cinf \rightarrow \cinf\; \textrm{such that} \; w|_{\bb D}=f \;\textrm{and which is}\\
            \textrm{compatible with $G$, i.e., $wGw^{-1}$ is again Kleinian}
        \end{array}\right\}
$$
Here $\cal S^0(\bb D)$ denotes the schlicht functions on the disc that are 1-point normalized at the origin. 
The map $$\pi_{BS}: \cal T_B(G) \rightarrow \cal T_S(G)\;, \qquad \pi_{BS}[\mu] = w^\mu|_{\bb D}\;,$$
mapping the Beltrami model to the schlich model is a bijection. First of all, $\mu \sim \nu$ implies $w^\mu_{|\bb D} \equiv w^\nu_{|\bb D}$, since if two holomorphic functions on the disc agree on $\partial \bb D$ they agree on the whole disc. By definition, the restrictions are quasiconformally extendable. The last ingredient is the fact that the restriction of the dilatation $\hat \mu(w):= \bar \partial w / \partial w$ of a quasiconformal homeomorphism $w$ of the Riemann sphere to $\bb D^c$ is in $\Belt(\bb D^c,G)$ iff $wGw^{-1}$ is Kleinian (see, e.g., \cite{TCAToTS} or \cite{UFaTS}).\\ \\
The reason these two models are of importance is the following: In the Beltrami model, points of Teichm\"uller space are quite hard to handle, since they are equivalence classes of measurable functions where the equivalence is defined utilizing the solution of a non-linear PDE. On the other hand, the Beltrami model induces a complex structure on Teichm\"uller space from the complex structure of the Banach space $L^\infty_{(-1,1)}(\bb D^c, G)$. In contrast, the points of the schlicht model are schlicht functions on the disc, and so all the machinery from complex analysis and geometric function theory is available to them. The price to pay is that there is no way of directly seeing the complex structure on Teichm\"uller space in this model.
\subsection{The Schwarzian Derivative}
The Bers embedding utilizes the schlicht model of Teichm\"uller space, where points are given by schlicht functions.
Bers' idea was to map Teichm\"uller space into another function space by applying a particular non-linear differential operator to the individual functions, namely the Schwarzian derivative.\\ \\
In order to define the Schwarzian, let $f:D \rightarrow
\bb C$ be locally injective and thrice differentiable on a domain $D \subset \bb C$, and set
$$S_f(z) : = \left(\frac{f''(z)}{f'(z)}\right)' -\frac{1}{2}\left(\frac{f''(z)}{f'(z)}\right)^2\;.$$
A simple computation shows that $S_f(z) \equiv 0$ on an open set iff $f \in \textrm{PSL}(2,\bb C)$, and that the chain rule for the Schwarzian of a composition is given by
\begin{eqnarray} \label{trafoSchw}
    S_{f\circ g} = (S_f \circ g)(g')^2 + S_g\;.
\end{eqnarray}
Hence $S_f = S_{1/f}$ since $1/f$ is of the form $M\circ f$ with $M\in \textrm{PSL}(2,\bb C)$, so we can
define the Schwarzian derivative on the set of locally injective meromorphic functions, which we denote by $\cal
M_{\mathrm{li}}(D)$. To do so, for $f \in  \cal
M_{\mathrm{li}}(D)$, set
$$S_f(p) = S_{\frac{1}{f}}(p)\;,$$ 
for all poles $p$ of $f$. Observe that local injectivity forces the pole
to be simple. The chain rule also allows us to define the Schwarzian acting on domains in $\cinf$ containing $\infty$ as follows: We set $\phi(z) = f(1/z)$ and define
$$ S_f(\infty) = \lim_{z\rightarrow 0} z^4 S_\phi (z)\;.$$
A very nice fact about the Schwarzian is the fact that the solution theory is understood very well (see, e.g., \cite{UFaTS}, Thm. II.1.1). 
\sloppy
\begin{prop} \label{SchwSol}
    For any simply connected domain $D \subset \cinf$, the Schwarzian derivative is a surjective operator \mbox{$S: \cal M_{\mathrm {li}}(D) \rightarrow \cal O(D)$}. Moreover, the solution $f$ to the equation $S_f = \phi$ is unique up to 				post-composition by M\"obius transformations.
\end{prop}
\fussy
In fact, a solution $f$ to the equation $S_f = \phi$ can be written as $f=h_1/h_2$, where $h_i$ are two linearly independent solutions to the linear equation $s'' + 1/2 \phi s =0$, and a change of basis in the solution space corresponds precisely to post-composing $f$ with a M\"obius transformation.\\ \\
Let $\cal{S}(D) \subset \cal M_{li}(D)$ denote the set of \emph{injective holomorphic} functions. Elements of $\cal S(D)$ are called \emph{schlicht} functions, or by some authors also \emph{univalent} functions. A famous result on schlicht functions that we need later is the following.
\begin{thm}[Koebe's $\frac 14$-theorem]\label{K1QThm}
		Let $f$ be a schlicht function on the disc $\bb D$ that fixes the origin and with $f'(0)$ = 1. Then the image 	 				$f(\bb D)$ contains the disc $\bb D_{\frac 14}$.
\end{thm}
To understand the image of the schlicht functions under the Schwarzian, we need to introduce some Banach spaces of holomorphic functions. Recall that any domain $D\subset \cinf$ such that the complement $\cinf \backslash D$ has at least three points admits a complete metric of constant negative curvature given by $\lambda^2_D |dz|^2$ and called the Poincar\'e metric. For example, for the unit disc, $\lambda^2_{\bb D}(z) = (1-|z|^2)^{-2}$.
\begin{dfn}
Let $D\subset \cinf$ be a domain that admits a Poincar\'e density $\lambda_D$. Then the space of \underline{bounded n-differentials}, denoted $B_n(D)$, is given by the holomorphic functions on $D$ for which the hyperbolic sup-norm
$$ \|f\|_{B_n(D)} := \mathrm{sup}_{z \in D} |f(z)| \lambda^{-n}_D(z)$$
is finite.
\end{dfn}
It is easy to see that this is a complex Banach space. A famous theorem by Kraus (rediscovered and often attributed to Nehari) bounds the $B_2$-norm of the Schwarzian of schlicht functions on the disc.
\begin{thm}[Kraus-Nehari Theorem]
Let f be a schlicht function on $\bb D$. Then
$$ |S_f(z)| (1-|z|^2)^2 \leq 6\;,$$
In particular, $\| S_f\|_{B_2(\bb D)} \leq 6$. This bound is sharp.
\end{thm}
It is easy to obtain a bound for any simply-connected hyperbolic domain $D\subset \cinf$ via a Riemann mapping $\psi:D\rightarrow \bb D$. Let $f \in \cal S(D)$, so $f\circ\psi^{-1}$ is schlicht on $\bb D$ and
\begin{eqnarray*}
    S_{f\circ\psi^{-1}}&=& (S_f \circ \psi^{-1})(\partial_z\psi^{-1})^{2}
    + S_{\psi^{-1}}\\
    &=& \left((S_f -S_\psi)\circ
    \psi^{-1}\right)(\partial_z\psi^{-1})^{2}\;,
\end{eqnarray*}
where the last step is due to the fact that the Schwarzian of an inverse can expressed by the Schwarzian of the map itself with the help of the chain rule \eqref{trafoSchw} as follows:
$$ 0 = S_{\psi \circ \psi^{-1}} = (S_\psi \circ \psi^{-1})(\partial_z\psi^{-1})^{2}+S_{\psi^{-1}}\;.$$ 
Now, multiplying by the Poincar\'e density, using its transformation behaviour and taking absolute values, we get
   $$ \| S_f\|_{B_2(D)} \leq 12 \qquad \forall \; f\in \cal S(D)\;.$$
There is a converse statement to this fact for a class of domains called \emph{quasidiscs}, which are images of the unit disc under a quasiconformal homeomorphism $\cinf$, and in fact, this property characterizes quasidiscs.
\begin{thm}[Gehring] \label{SchlichtCrit} 
   Let $D\subset\cinf$ be a simply-connected hyperbolic domain. Then $D$ is a quasidisc iff there exists a constant 						 $\delta>0$ such that for all $f \in \cal M_{\mathrm{li}}(D)$, 
   $\|S_f\|_{B_2(D)} \leq \delta$ implies that $f \in \cal S(D)$.
\end{thm}
A proof can be found in \cite{UFaTS}, Section II.4.6.\\ \\
Recall that in the previous section we defined a model of the Teichm\"uller space of a Fuchsian group $G$ within the space of schlicht functions.  If $G=\bb 1$, $\cal T_S(\bb 1)$ is called the \emph{universal Teichm\"uller space}, and its image under the Schwarzain in $B_2(\bb D)$ will be denoted by $T(\bb 1)$. Also, let us denote the image of all schlicht functions on the disk under the Schwarzian by $\bb S$. The following highly interesting theorem can be found in \cite{TCAToTS}.
\begin{thm} [Gehring] \label{IntClos}
    The interior of $\bb S$ is precisely $T(\bb 1)$, but the closure of $T(\bb 1)$ is a proper subset of
    $\bb S$.
\end{thm}
Let us only remark that the first claim follows quite easily from Theorem \ref{SchlichtCrit} while the second, more surprising statement is much harder. Gehring succeeded in the proof by explicitly constructing domains $D$ given by $\cinf \backslash \gamma$, where $\gamma$ are special spiral arcs which have in some sense the opposite property: there exists a $\delta > 0$ such that for all schlicht functions $f$ on $D$ with $\|S_f\|_{B_2(D)}< \delta$, $f(D)$ is \emph{not a Jordan domain}. The theorem then follows easily.
\subsection{The Bers Embedding}
The schlicht model of Teichm\"uller space is a model in the the category of sets, which is yet not very satisfactory. The Bers embedding realizes this model as a domain in a complex Banach space.
\begin{dfn}\label{BersEmbeddingDef}
   Let $G$ be a Fuchsian group acting on $\bb D$. The \underline{Bers em-} \underline{bedding} $\beta$ of Teichm\"uller space is         given by
   $$ \beta: \cal T_B(G) \rightarrow B_2(\bb D, G)\;, \qquad \beta = S \circ \pi_{BS}\;.$$
   The image $\beta(\cal T_B(G))$ will be denoted by $T(G)$. The lift of this map,
   $$\tilde \beta : \bb B_1\left(L^\infty_{(-1,1)}(\bb D^c, G)\right) \rightarrow  B_2(\bb D, G)\;, \qquad \tilde \beta := \beta \circ \pi_T\;,$$ is 				 called the \underline{Bers projection}.
\end{dfn}
The target spaces in the theorem have not been introduced yet. As the notation suggests, $B_n(\bb D,G)$ is a subset of $B_n(\bb D)$. It consists of precisely those functions that satisfy
$$ (f\circ g)(z) (g')^n(z) = f(z) \qquad \forall g\in G, z \in \bb D\;.$$
That this indeed is the target is a special case of Proposition \ref{DefQDCorr} below. Geometrically, this space can be identified with the space of holomorphic sections of the $n$-th tensor power of the canonical bundle of the surface $\Sigma \cong \bb D/G$ in the cocompact case. In the case of a punctured surface $\Sigma$ of type $(g,p)$ one has to consider twisting the canonical bundle with a divisor coming from the punctures. The dimension of this space is well-known to be (see, e.g., \cite{AFaKG}) 
$$ \mathrm{dim}_{\bb C} B_n(\bb D, G) = (2n-1)(g-1) + pn\;.$$
These cases correspond to Fuchsian groups of first kind (in case the group of first kind and has elliptic elements there is a further finite term coming from the ramifications, but we don't need the explicit formula in what follows). For Fuchsian groups of second kind, so in particular for $G=\bb 1$, the space is infinite dimensional.\\ \\
But we still need to justify calling the map $\beta$ an \emph{embedding}. The following beautiful theorem is due to Bers. 
\begin{thm} \label{BembThm}
    Let $G$ be a Fuchsian group acting on $\bb D$. The Bers projection $\tilde \beta$
    is a holomorphic submersion and factors precisely through $\pi_T$ yielding a holomorphic embedding of $\cal T_B(G)$     into $B_2(\bb D, G)$ as an open, bounded and contractible domain.
\end{thm}
For a detailed exposition and full proof we refer to the textbook \cite{TCAToTS}. On the other hand, the proofs of our main theorems later, which generalize this theorem considerably, reprove everything except the injectivity and contactibility.\\ \\
The theorem has a wealth of consequences. Recall that by Theorem \ref{SchlichtCrit}, there exists a constant $\delta$ such that the $\delta$-ball in $B_2(\bb D)$ around the origin is contained in $\bb S$, hence the $\delta$-ball around the origin in $B_2(\bb D, G)$ is also contained in $\bb S \cap B_2(\bb D, G)$. Further, by Theorem \ref{IntClos}, 
$$ T(G) = T(\bb 1) \cap B_2(\bb D, G) = \textrm{int}(\bb S) \cap B_2(\bb D, G) \supset \bb D_\delta\;,$$
so $ \mathrm{dim}_{\bb C} T(G) = \mathrm{dim}_{\bb C} B_2(\bb D, G)$.\\ \\
Let us say a few more words about the geometry of images $T(G)$ for they are very intriguing and not yet fully understood. First of all, the constant $\delta$ in Theorem \ref{SchlichtCrit} depends on the quasidisc $D$; for $D=\bb D$ it is well-known to have the value $2$. In fact, for bounded differentials of norm less than two, one can write down the inverse of the lifted Bers embedding explicitly. This goes by the name of the Ahlfors-Weill section (see below).\\ \\
However, observe that if we restrict to the intersection $\bb S \cap B_2(\bb D, G)$ there could be a bigger ball around the origin. In general, let us define the two quantities,
$$ i(G) := \sup_{\delta \in \bb R} \{\bb D_\delta \subset T(G)\}\;, \qquad o(G):= \inf_{\delta \in \bb R} \{ \bb D_\delta \supset T(G)\}\;,$$
called the \emph{inradius} resp.~ the \emph{outradius} of the Teichm\"uller space $T(G)$. By the facts on schlicht functions, we have $i(G)\geq 2$ and $o(G)\leq 6$. The following facts concerning the in- resp.~ outradius can be found in \cite{OIRoTS} resp.~ \cite{OtOotTS}: $i(G)$ is strictly greater than two for any finitely generated Fuchsian group $G$ of first kind but there exists a sequence of quasiconformal deformations\footnote{A quasiconformal deformation of a group $G$ is given by $G':=wG w^{-1}$ where $w$ is a quasiconformal homeomorphism of $\cinf$ the dilatation of which (restricted to $\bb D$) is in $L^\infty_{(-1,1)}(\bb D,G)$.} $\{G_i\}$ of $G$ such that $i(G_i)\rightarrow 2$ for $i\rightarrow \infty$. $o(G)$ equals $6$ for Fuchsian groups of second kind and $o(G)$ is strictly less than $6$ for finitely generated Fuchsian groups of first kind. Yet given a finitely generated Fuchsian group of first kind $G$, and again, there exists a sequence $\{G_i\}$ of quasiconformal deformations of $G$ such that $o(G_i)\rightarrow 6$ for $i\rightarrow \infty$. Beware, however, that these facts don't give much information on $\bb S(G) := \bb S \cap B_2(\bb D,G)$ since $\mathrm{cl}\: (T(G)) \neq \bb S(G)$.\\ \\
Now although the inradii of Teichm\"uller spaces are always greater or equal to two, it is perhaps a little bit surprising that a quasiconformal homeomorphism in the equivalence class corresponding to any $\phi \in B_2(\bb D,G)$ of norm $<2$ can be constructed explicitly. By taking its dilatation we obtain the so-called Ahlfors-Weill section $s$, which is a section of the projection $\tilde \beta = \beta \circ \pi_T$.
\begin{thm}\label{AhlWeil}
   The map
   $$ s: B_2(\bb D, G) \rightarrow L^\infty_{(-1,1)}(\bb D^c, G)\;, \qquad \phi(z) \mapsto -\frac{1}{2} \phi(1/\bar
   z)(1-|z|^2)^2\bar z^4\;,$$
   is a holomorphic right inverse of the Bers projection when restricted to the open ball of radius 2 in $B_2(\bb D, G)$.
\end{thm}
Let us point out here that this section is real-analytic. This implies in particular that the normalized solutions $w^{s(\phi)}$ are real analytic in $\bb D$ and $\bb D^c$. Also, as one would expect at this point, there exists an Ahlfors-Weill section for any quasidisc $D$ (see, e.g., \cite{TCAToTS}, Sect. 3.8.3) \\ \\
Maybe the most important feature of the Bers embedding (or more precisely, of the images of Teichm\"uller spaces under the Bers embedding) is that boundary points of Teichm\"uller space have extrinsic meaning, and this leads to a geometric understanding of degenerations of Riemann surfaces resp.~ Fuchsian groups. Since this is a further motivation for our construction of holomorphic mappings in the spirit of the Bers embedding, we want to present this enlarged framework briefly. The notion of the deformation space of a Fuchsian group was introduced by Kra in \cite{DoFG}, \cite{DoFG2}.
\begin{dfn}
   Let G be a Fuchsian group acting on $\bb D$. A \underline{deformation} of G is a pair $(\chi,f)$ where 								 $\chi:G\rightarrow \mathrm{PSL}(2,\bb
   C)$ is a homomorphism and $f: \bb D \rightarrow \cinf$ a locally injective meromorphic function that satisfies the
   compatibility equation $$ f \circ g = \chi(g) \circ f\;, \qquad \forall g \in G\;.$$
   Two deformations $(\chi_1, f_1)$ and $(\chi_2, f_2)$ are called equivalent iff
   $$ \exists M \in \mathrm{PSL}(2, \bb C): \qquad f_2 = M \circ f_1\:, \qquad \chi_2(g) = M \circ \chi_1(g) \circ
   M^{-1} \;.$$
   The set of all equivalence classes of deformations of $G$ is denoted by $\mathrm{Def}(G)$ and is called the 						 \underline{deformation space} of 		 	 the Fuchsian group $G$.
\end{dfn}
Observe also that the data of the definition is somewhat redundant: $f$ determines $\chi$, since by local injectivity, for any point $z \in f(\bb D)$ there is a neighborhood $U_z$ where $f$ is invertible. Hence
$$ \chi(g) (w) = (f \circ g \circ f^{-1})(w) \qquad \forall \: w \in U_z\;,$$
and this determines $\chi(g)$ completely, since a M\"obius transformation is characterized by its value on three points. But also conversely, any $f \in M_{\mathrm{li}}(D)$ \emph{is} the developing map of a deformation as we will see in the proof of \ref{DefQDCorr}. With this in mind, we will often identify functions in $\cal M_{\textrm{li}}(\bb D)$ with the corresponding deformations $(f, \chi)$ they
induce. For clarity in later use we will denote the forgetful maps by 
\begin{eqnarray*} 
\begin{array}{ll}
   \mathrm{dev}: \textrm{Def}(G) \rightarrow \cal M^0_{\textrm{li}}(\bb D)\;,& \qquad \mathrm{dev}([f,\chi]) = \tilde f\\
   \mathrm{hom}: \textrm{Def}(G) \rightarrow \mathrm{Hom}(G,\mathrm{PSL}(2,\bb C))\;,& \qquad \mathrm{hom}([f,\chi]) = \tilde \chi\;.
\end{array}
\end{eqnarray*}
Let's return to the deformation spaces of Fuchsian groups. These will now be given the structure of a complex vector space.
\begin{prop}\label{DefQDCorr}
   There following map is bijective and hence induces a $1-1$ correspondence,
   $$c_G: \mathrm{Def}(G)\rightarrow\cal Q(G):=\{f \in \cal O(\bb D): f = (f\circ g) (g')^2 \; \forall \: g \in G\}    \qquad c_G:= S \circ \mathrm{dev}\;,$$
   between the equivalence classes of deformations of $G$ and the quadratic differentials $\cal Q(G)$ for $G$. 
\end{prop}
\begin{proof} In Theorem \ref{SchwSol} we established the correspondence $\mathrm{Def}(\bb 1) \cong M^0_{li}(\bb D)$ with $\cal Q(\bb 1) \cong \cal O(\bb D)$. As a next step, we show that $c_G([f,\chi]) = S_{\mathrm{dev}([f,\chi])}$
is a quadratic differential for $G$ whenever $[f,\chi] \in \mathrm{Def}(G)$. Recall the transformation behaviour
(\ref{trafoSchw}) for $g \in$ PSL$(2,\bb C)$,
$$ S_{f\circ g} = (S_f \circ g)(g')^2\;, \qquad S_{g\circ f} = S_f\;.$$
If $g \in G$ we have $f\circ g =  \chi(g) \circ f$ so altogether we
get $$ S_f = (S_f \circ g)(g')^2\;,$$ i.e. the Schwarz\-ian of the
developing map behaves like a quadratic differential for $G$. Now let $f \in \cal M_{\mathrm{li}}(\bb D)$ be a funcion such that $S_f = \phi $ is a quadratic differential for the group $G$. Then for all $g \in G$,
$$ S_{f \circ g} = (S_f \circ g)( g')^2 = (\phi\circ g) ( g')^2 = \phi\;,$$ 
and hence by the uniqueness part of the solution theorem of the Schwarz\-ian differential equation there exists a M\"obius transformation $\chi(g)$ such that $\chi(g) \circ f = f \circ g$. This association is a homomorphism,
$$ \chi(g_1g_2) \circ f = f \circ (g_1g_2) = (f \circ g_1)\circ g_2 = \chi(g_1) \circ f \circ g_2 = \chi(g_1)\chi(g_2) \circ f\;,$$
and the pair $(\chi,f)$ therefore a deformation of $G$. 
\end{proof}
Observe that one can consider $\cal T_S(G)$ naturally as a subset of $\mathrm{Def}(G)$ by identifying $f$ with $\mathrm{dev}^{-1}(f)$. But to emphasize once more, any point in $\cal Q(G)$, and hence any point in $B_2(\bb D,G)$, corresponds to a Fuchsian group by Proposition \ref{DefQDCorr}. In particular, points on $\partial T(G)$. The Fuchsian groups on $\partial T(G)$ consist of \emph{(partially) degenerate groups} and \emph{regular b-groups}, the latter corresponding geometrically to \emph{noded Riemann surfaces} (see \cite{OBoTSaoKG:1}, \cite{OBoTSaoKG:2} and \cite{OBoTSaoKG:3}). Including only the latter, one arrives at a nice partial completion of Teichm\"uller space called \emph{augmented Teichm\"uller space} \cite{DFoRS} on which the mapping class group operates by homoeomorphisms and the quotient is homeomorphic to the Deligne-Mumford compactification of moduli space. All of this should constitute enough motivation to study other holomorphic maps of Teichm\"uller spaces as we do from Section \ref{HMoTS} onward.
\subsection{Banach spaces of holomorphic functions}\label{BSoAF}
At several points later we will need Banach spaces of automorphic forms. And also, in the proof of Theorem \ref{KernelHBM} we need several of the main results from this theory. In this section we briefly give the definitions and state the theorems we need later on. All material can be found, e.g., in \cite{AFaKG}.\\ \\
$L^p(D)$ will as usual denote the space of $p$-integrable
measurable functions on $D$, which can be any hyperbolic open set of $\cinf$,
and $\| \cdot \|_p$ will denote the $p$-norm. If a group $G$ acts on $D$, a factor of automorphy for the $G$-action on $D$ is a map $\rho_s: G\times D \rightarrow \bb C^*$ such that for fixed $g \in G$, \mbox{$\rho_s(g, \cdot):D\rightarrow \bb C^*$} is holomorphic and that for all $g_i \in G$, $\rho_s(g_1g_2,z) = \rho_s(g_1, g_2(z))\cdot \rho_s(g_2,z)$. A factor is an $s$-factor, iff $|\rho_s(g,z)|=|g'(z)|^{-s}$. For an $s$ factor $\rho_s$ of a Kleinian group $G$ acting properly discontinuously on $D$, $L_{\rho_s}(D,G)$ will
denote the space of measurable automorphic forms, i.e., functions that satisfy $(f\circ g)(z) = f(z) \rho_s(g,z)$. On $L_{\rho_s}(D,G)$ the following expressions are well-defined norms,
\begin{eqnarray*}
   \|f\|_{L^p_s(D,G)} &:=& \left(\int_{\cal F} \lambda_{D}^{2-ps}(z)|f(z)|^p
d^2z\right)^{1/p}, \qquad 1 \leq p < \infty\;, \\
   \|f\|_{L^\infty_s(D,G)} &:=& \essup_{z \in \cal F} \left \{
\lambda_{D}^{-s}(z) |f(z)| \right \}\;.
\end{eqnarray*}
and the set of functions for which the norm is finite is denoted by $L^p_{\rho_s}(D,G)$. These are Banach spaces.
Note that the integrals are not performed over $D$ but only over a
fundamental domain $\cal F$ for the action of $G$ on $D$. The subspace of holomorphic automorphic forms is denoted by \begin{eqnarray*}
   A^p_{\rho_s}(D,G)&:=& \cal O(D) \cap L^p_{\rho_s}(D,G)\;.
\end{eqnarray*}
In case $\infty \in D$ we will have to add the extra assumption\footnote{Otherwise the constructions will not be independent of the domain, e.g., the operation of pull-back introduced below would map functions holomorphic around $\infty$ to \emph{meromorphic functions} at the origin (see \cite{AFaKG} for details).} that $f \in O(|z|^{-2s})$. Also, since $L^p$-convergence of holomorphic functions implies local uniform convergence, these subspaces are closed and hence Banach spaces themselves.\\ \\ 
We follow the tradition of denoting the spaces $A^\infty_{\rho_s}(D,G)$ by $B_{\rho_s}(D,G)$ and the spaces $A^1_{\rho_s}(D,G)$ by $A_{\rho_s}(D,G)$. Also, if we are dealing with an integer power $n$ of the canonical factor of automorphy, i.e., the factor $\rho_n(g,z) = g'(z)^{-n}$, we will use the subscript $n$ instead of $\rho_n$. And finally, when $G = \{\bb 1\}$ is the trivial group, we
simplify the notation and write $L^p_{s}(D)$ instead of $L^p_{\rho_s}(D,\bb 1)$, and similarily for the holomorphic subspaces. \\ \\
These complex Banach spaces are of course independent of the chosen uniformization or, so to speak, invariant under
conjugation. More precisely, let $D$ be a simply-connected hyperbolic domain and $\phi:D \rightarrow \phi(D)$ be a biholomorphism. Then $\phi$ induces norm preserving isomorphisms for $1\leq p \leq \infty$, called the pull-back,
$$ \phi_s^*: L^p_{\rho_s}(\phi(D),\phi G \phi^{-1})\rightarrow
L^p_{\rho_s}(D,G)\;,\qquad (\phi_s^*f) (z) = (f\circ \phi) (z) \phi' (z)^s\;, $$ which respect the subspaces
$A^p_{\rho_s}$ of holomorphic functions.\\ \\
For conjugate numbers, i.e., $1/p + 1/p'=1$, we can introduce the product, called the \emph{Weil-Petersson pairing},
$$ L^p_{\rho_s}(D,G) \times L^{p'}_{\rho_s}(D,G)
\rightarrow \bb C\;, \qquad \langle f,g \rangle^G_s := \int_{\cal F} f(z) \ol {g(z)} \lambda^{2-2s}_D(z) d^2z \;.$$
\sloppy The integral is seen to converge by rewriting the integrand as follows,
$$f(z) \ol{g(z)} \lambda^{2-2s}_D(z) = \left(f(z) \lambda^{\frac{-sp}{p}}_D(z)\right) \cdot \left(\ol{g(z)} \lambda^{\frac{-sp'}{p'}}_D(z)\right) \cdot \lambda_D^2(z)\;,$$
and then applying H\"older's inequality with respect to the measure $\lambda^2_D(z) d^2z$. This also establishes the fact that the Weil-Petersson pairing induces an isometric isomorphism 
\begin{eqnarray} \label{IsoIso}L_{\rho_s}^{p'} (D,G) \cong \left(L_{\rho_s}^{p} (D,G)\right)^*\end{eqnarray}
for any $1\leq p<\infty$ and any group $G$ acting on $D$. It is remarkable that the subspaces of holomorphic automorphic forms respect this duality, albeit not isometrically.\fussy
\begin{thm}[\cite{AFaKG}] \label{NondegenerateOnAs}
  For $1\leq p<\infty$ the anti-linear map
  $$  A_{\rho_s}^{p'} (D,G) \rightarrow \left(A_{\rho_s}^{p} (D,G)\right)^*\;, \qquad \psi \mapsto l_\psi:=\langle \cdot, \psi\rangle^G_s\;,$$
  is an isomorphism which satisfies the norm inequality 
  $$ c_s^{-1} \| \psi \|_{L^{p'}_s(D,G)} \leq \| l_\psi \| \leq \|\psi\|_{L^{p'}_s(D,G)}\;.$$
\end{thm}
The projection operator from the space of measurable automorphic forms to the holomorphic automorphic forms can be written down quite explicitly. For this, let
$$ k_{\bb D}: \bb D \times \bb D \rightarrow \bb C\;, \qquad k_{\bb D}(z,w) := \frac{1}{\pi(1-z\bar w)^2}\;.$$
be the well-known Bergman kernel on the disc. On any other domain related to $D$ related to $\bb D$ via a biholomorphism $\psi:D\rightarrow
\psi(D)$, the kernel is given by
$$ k_{\psi(D)} (\psi(z),\psi(w)) \psi'(z)\ol{\psi'(w)} = k_{D} (z,w)\;.$$
Now let us define a related function, 
\begin{eqnarray} 
	\label{BergmanDef} K_{D,s}(z,w):= (2s-1)\pi^{s-1}\left(k_D(z,w)\right)^s\;, \qquad c_s:=\frac{2s-1}{s-1}\;,
\end{eqnarray}
which we call the $s$-Bergman kernel. 
\begin{thm} \label{last}
    The operator defined by the expression
    $$ f\mapsto (\beta_{\rho_s}f)(z):= \int_{D} \lambda^{2-2s}_D(w)K_s(z,w)f(w) d^2w\;,$$
    is a well-defined projection operator $L^p_{\rho_s}(D,G)\rightarrow A^p_{\rho_s}(D,G)$ of norm at most $c_s$.           Moreover it is symmetric with respect to the Weil-Petersson pairing, i.e.,
     $$\langle \beta_{\rho_s} f,g\rangle_s^G = \langle f,\beta_{\rho_s} g\rangle^G_s\;.$$
\end{thm}
The final theorem we need later on concerns the normal convergence of the following series called the \emph{Poincar\'e series},
$$\Theta_{\rho_s} [f] (z):= \sum_{g \in G} f(g z) \rho_{g}(z)^{-1}\;,$$
for given $s$-factor of automorphy $\rho_s$. 
\begin{thm}\label{PS}
   If $\|f\|_{L^1_s(D)} < \infty$, the series $\Theta_{\rho_s}[f]$
converges normally. Moreover $\Theta_{\rho_s} [f] \in
L^1_{\rho_s}(D,G)$ and $$ \| \Theta_{\rho_s} [f] \|_{L^1_{\rho_s}(D,G)} \leq \|f\|_{L^1_s(D)}\;.$$
\end{thm}
This means that the Poincar\'e operator $\Theta_{\rho_s}$ is a bounded linear operator $L^1_s(D)\rightarrow L^1_{\rho_s}(D,G)$. By normal convergence its restriction to $A^1_s(D)$ maps into $A^1_{\rho_s}(D,G)$. The Poincar\'e operator is compatible with the Weil-Petersson product in the following way, which can be checked by straight-forward calculation.
\begin{lemma} \label{scalar}
   Let $f \in L_{\rho_s}^\infty(D,G)$, $g \in L^1_{\rho_s}(D,G)$ and $g=\Theta_{\rho_s}[h]$ for some function
$h \in L_s^1(D)$. Then the scalar product can be computed by
$$ \langle f,g\rangle^G_s = \int_D f(w) \ol{h(w)} \lambda_D^{2-2s}(w) d^2w = \langle f, h\rangle_s^{\bb 1}\;.$$
\end{lemma}
\section{Holomorphic Maps of Teichm\"uller Spaces}\label{HMoTS}
Now that we have introduced all the needed background, we immediately come to the main theorem, which is the starting point of all investigations in this paper, in its precise formulation. 
\begin{thm}\label{higherBersholomorphic}
	  Let $Q:\cal S(D) \rightarrow \cal O(D)$ be a differential operator that satisfies 
	  $$ Q[f\circ g] = (Q[f]\circ g)(g')^{m} \qquad \forall g \in \mathrm{PSL}(2, \bb C)\;,$$
	  for some $m \in \bb Z$ and such that $Q[f]$ is a polynomial in $f'',\ldots, f^{(N)}$ and $(f')^{-1}$ with complex coefficients. Then $Q$ induces a holomorphic map 
	  $$\beta^Q:\cal T_B(G) \rightarrow B_{m}(\bb D)\;, \qquad \beta^Q([\mu]) = Q(\pi_{BS}([\mu]))\;,$$ for any Fuchsian group $G$.
\end{thm}
\emph{Proof}. The map is well-defined, since $w^{\mu}$ only depends on the class $[\mu]$. Recall also that $\pi_{BS}([\mu])$ is given by the 1-point normalization at the origin of $w^\mu|_{\bb D}$. To prove holomorphicity for all $G$ it clearly suffices to prove it for $G=\bb 1$ since all $T(G) \subset T(\bb 1)$ are complex submanifolds.\\ \\
The complex structure on $\cal T_B(\bb 1)$ is inherited from the one on $L^\infty(\bb D^c)$, so if we lift the map $\beta^Q$ to
$$\tilde \beta^Q:= \beta^Q \circ \pi_T\;, \qquad \tilde \beta^Q(\mu) = \beta^Q([\mu])\;,$$
then $\beta^Q$ is holomorphic iff $\tilde \beta^Q$ is. In general, a map from $\bb C$ into a complex Banach space is said to be holomorphic iff the Gateaux derivative exists and is finite. A map from an infinite-dimensional complex Banach space into a complex Banach space is holomorphic iff it is locally bounded and it is holomorphic when restricted to any finite-dimensional subspace, which again is true iff it is holomorphic when restricted to any \emph{one-dimensional} subspace by Hartog's theorem. Hence we have to show that $\tilde \beta^Q$ is locally bounded and that
$$ \lim_{t\rightarrow 0}\frac{\| \tilde \beta^Q (\mu + t \nu) - \tilde \beta^Q(\mu)\|_{B_{m}(\bb D)}}{t}$$
exists and is finite for all $\mu \in \textrm{Belt}(\bb D^c, \bb 1)$ and $\nu \in L^\infty(\bb D^c)$, which we do in two separate lemmas. 
\begin{lemma}\label{BetaLocallyBounded}
		The function $\tilde \beta^Q(\mu + t \nu)(z_0 + z)$, viewed as a function of $(z,t)$ in a small neighborhood of the origin in $\bb C^2$, is locally bounded for any $z_0 \in \bb D,$ $\mu \in \mathrm{Belt}(\bb D^c)$ and $\nu \in L^\infty(\bb D^c)$.
\end{lemma}
\begin{proof} The function is well defined on 
$$\left\{z:\:|z + z_0| < 1\right\} \times \left\{t:\:|t| < (1-\|\mu\|_\infty)\|\nu\|_\infty^{-1}\right\}\;.$$
We restrict it to a product of discs, or more precisely to $\bb D_r (z_0) \times \bb D_\epsilon$ where $r:=\textrm{dist}(z_0, \partial \bb D)$ and $\epsilon$ is any real number such that the norm of $\mu + t \nu$ is bounded by $K<1$  for $t \in \bb D_\epsilon$. Let $M$ be the M\"obius transformation obtained by composing the translation $z\mapsto z - z_0$ with the dilatation $z \mapsto r^{-1} z$. The disc $\bb D_r (z_0)$ is mapped onto the unit disc by $M$. Define the compositions 
$$ \mu^\bullet := \mu \circ M \;, \qquad \nu^\bullet := \nu \circ M\;.$$
Now in general, the Beltrami differential of a composition is given by 
\begin{eqnarray*} \label{ChainRuleBeltramiDifferential}
\hat \mu(g\circ f) = \frac{\hat \mu(f) + (\hat \mu(g) \circ f)\cdot(\overline{\partial f} /\partial f)}{1+(\hat \mu(g) \circ f)\cdot\overline{\hat \mu(f)}\cdot(\overline{\partial f}/\partial f)}\;,
\end{eqnarray*} 
and hence, because $\hat \mu(M) = 0$, we have
$$ \hat \mu ( w^{\mu + t \nu}\circ M) = \hat \mu(w^{\mu + t \nu}) \circ M \frac{\overline {M'}}{M'} = (\mu + t \nu)\circ M \frac{\overline {M'}}{M'}\;.$$
On the other hand,
\begin{eqnarray*}
\bar \partial (w^{\mu + t \nu} \circ M) &=& \bar \partial w^{\mu + t \nu} \circ M \cdot \overline{M'}\\
\partial (w^{\mu + t \nu} \circ M) &=& \partial w^{\mu + t \nu} \circ M \cdot M'\;,
\end{eqnarray*}
so $w^{\mu + t \nu} \circ M =: w^{\mu^\bullet + t\nu^\bullet}$ solves the Beltrami equation (this is a slight abuse of notation which will only be used in this proof: By previous conventions, $w^{\mu^\bullet + t\nu^\bullet}$ should be used to denote the \emph{1-point normalized} solution of the Beltrami equation for the coefficient $\mu^\bullet + t\nu^\bullet$, and in general, this is not the same as $w^{\mu + t \nu} \circ M$) for the coefficient $\mu^\bullet + t \nu^\bullet$. We remarked in Section \ref{Bersemb} that the family 
$$\cal F_\nu:= \{w^{\mu + t \nu}\:: \: t \in \bb D_\epsilon\}$$ 
is a normal family of K-qc mappings. Such a family is also equi-H\"older continuous on compact sets (\cite{QA}, Ch. II.5),
$$ |f(z_1)-f(z_2)| \leq C |z_1-z_2|^{1/k}\;, \qquad \forall z_1,z_2 \in K \subset \bb C \; \textrm{and}\; f \in \cal F_\nu\;,$$
which results in a bound
$$ |f(z)| \leq C' \qquad \forall \:z \in \bb D_r(z_0)\;.$$
This also implies the same bound on the values of the functions of the family $\cal F^\bullet_\nu:=\{w^{\mu^\bullet + t \nu^\bullet}\:: \: t \in \bb D_\epsilon\}$, and this bound on the family on the boundary $\partial \bb D$ of the unit disc yields a bound on all the derivatives at the origin by the Cauchy estimates:
$$ \left|\frac {d^m}{dz^m}w^{\mu^\bullet + t \nu^\bullet}(0)\right| \leq C_m\;, \qquad \forall\: t \in \bb D_\epsilon\;.$$
Now because $Q[f]$ is a polynomial in the first $N$ derivatives of $f$ and $(f')^{-1}$, the bounds $C_m$ induce a bound on the value of the image of the operator $Q$ applied to $f$ at the origin,
$$|Q[f](0)| \leq N_m\;, \qquad \forall f \in \cal F_\nu\;.$$
By the transformation behaviour with respect to precomposition with a disc automorphism, 
$$ Q[w^{ \mu^\bullet + t \nu^\bullet}](0) = Q[w^{\mu + t \nu}\circ M](0) = Q [w^{\mu + t \nu}](z_0) \cdot r^{m}\;,$$
so we get a pointwise estimate for the expression in the $B_{m}(\bb D)$-norm,
$$ \lambda_{\bb D}^{-m}(z_0) |\tilde \beta^Q(\mu + t \nu)(z_0)| \leq \lambda_{\bb D}^{-m}(z_0) r^{-m} N_n\;.$$
If we combine this with a well-known and fundamental estimate on the Poincar\'e density \cite{AFaKG},
\be \label{FEoPD}
   1\geq \lambda_{\bb D}(z_0) \textrm{dist}(z_0, \bb D) \geq \frac{1}{4}\;,
\ee
(where the latter one is only valid for domains not containing infinity) and $r \leq \textrm{dist}(z_0, \bb D)$ we arrive at
$$ \lambda_{\bb D}^{-m}(z_0) |\tilde \beta^Q(\mu + t \nu)(z_0)| \leq 4^{m} N_m\;,$$
i.e., the norm of $\tilde \beta^Q$ is a locally bounded function. Observe that this works because the power of $r$ and $\lambda_{\bb D}$ are precisely the same. In any other case, there would be no uniform bound. \end{proof}
\begin{lemma}
	The Gateaux derivative of $\tilde \beta^Q(\mu + t \nu)$ as a function of $t$ exists at $t=0$ and is finite for all $\mu \in \mathrm{Belt}(\bb D^c,   \bb 1)$ and $\nu \in L^\infty(\bb D^c)$.
\end{lemma}
\begin{proof}
Let us abbreviate the function $\tilde \beta^Q(\mu + t \nu)(z)$ by $\phi(t,z)$, and let $\epsilon$ be as in the proof of the previous lemma. For fixed $z$, this is a holomorphic function of $t$, since $w^{\mu + t \nu}(z)$ is holomorphic in $t$ and $Q$ leaves the regularity of the $t$-dependence unaltered because of its polynomial structure. By Cauchy's integral formula we can estimate for $|t| < \epsilon$
$$ |\phi(t,z) - \phi(0,z)|\leq \frac{1}{2 \pi} \textrm{sup}_{|\eta| = \epsilon}|\phi(\eta,z)| \int_{|\eta|= \epsilon} \left|\frac{1}{\eta - t }-\frac{1}{\eta}\right| d\eta\;.$$
Now, by Lemma \ref{BetaLocallyBounded} the quantity $\textrm{sup}_{|\eta| = \epsilon}|\phi(\eta,z)|$ is locally (independent of $z$) majorized by $4^{m} N_m$. Moreover, we are interested in $\phi(t,z)$ near the origin, so we can restrict $t$ to the disc $\bb D_{\frac \epsilon 2}$. The following estimate,
$$  \int_{|\eta|= \epsilon} \left|\frac{1}{\eta - t_1}-\frac{1}{\eta-t_2}\right| d\eta \leq \frac {8 \pi}{\epsilon} |t_1 - t_2|\;,$$
is straightforward for $t_i \in \bb D_{\frac \epsilon 2}$, and we use it to obtain
\begin{eqnarray}
	\label{dsfdsf} |\phi(t,z) - \phi(0,z)|\leq 2 \cdot 4^{m+1} N_m \frac{|t|}{\epsilon^2}\;,
\end{eqnarray}
which in other words says that $\tilde \beta_n$ is locally Lipschitz. But we want a little more. For this, let us denote the difference quotient of $\phi(t,z)$ at $t=0$ by $\psi(t,z)$. This can be estimated in the same way with the help of \eqref{dsfdsf},
\begin{eqnarray*}
	|\psi(t_1,z)- \psi(t_2, z)| &\leq& \frac{1}{2\pi}\int_{|\eta|= \epsilon}|\psi(\eta,z)|\left|\frac{1}{\eta-t_1}- 	\frac{1}{\eta - t_2}\right|d\eta\\
	&\leq&\frac{2 \cdot 4^{m+2} N_m}{\epsilon^3} |t_1 - t_2| \;.
\end{eqnarray*}
Hence taking a sequence $t_i \rightarrow 0$, the sequence 
$$t_i^{-1}\left(\tilde \beta^Q(\mu + t_i \nu) - \tilde \beta^Q(\mu)\right)\;,$$ is a Cauchy sequence in the Banach space $B_{m}(\bb D)$ and therefore converges to a unique definite element in $B_{m}(\bb D)$. This proves the lemma and concludes the proof of the theorem. 
\end{proof}
One might wonder at this point if there really exist differential operators satisfying the prerequisites of the theorem. We will see in Section \ref{HBM} that there indeed are, and we will study the holomorphic maps they induce in quite some detail. We also remark that \cite{MIOoRS} contains a rather complete classification of operators satisfying the prerequisites of the theorem.\\ \\
To understand the mappings in more detail, the next step is to look at their infinitesimal behaviour. To this end, let us determine a general formula for their differential at the origin.
\begin{thm}\label{ThmDifferential}
   Let $Q$ be as in Thm.~ \ref{higherBersholomorphic}, and let 
   $$M_1(Q[f])= \sum_{k,l} a_{k,l} \frac{f^{(k)}}{(f')^l}\;,$$
   be the part of the polynomial $Q[f]$ which consists of monomials of degree one in $f'', \ldots, f^{(N)}$.
   Then the derivative $D_0\tilde \beta^Q$ at the origin of $L^\infty(\bb D^c)$ is given by the bounded linear operator
   $$ D_0\tilde \beta^Q: L^\infty(\bb D^c) \rightarrow B_{m}(\bb D)\;, \quad \nu \mapsto \sum a_{k,l} \frac {(-1)^{k}k!}{\pi} \int_{\bb D^c} \frac{\nu(\eta)}{(z-\eta)^{k+1}}d^2\eta\;.$$
\end{thm}
\begin{proof}
Let $w$ be the coordinate on $\bb D$.  The quasiconformal solution to the trivial Beltrami differential 1-point-normalized at $0$ is $f(w,0) = w$, which we will simply denote by $f(w)$. Fix $\nu \in L^\infty(\bb D^c)$. The 1-point-normalized solutions to the Beltrami equation for $t\nu$ with $t \in \bb D_{1/\|\nu\|}$ will be denoted by $f(w,t)$. Further, let $\phi(w,t):= Q[f(w,t)]$ denote the image in $B_{m}(\bb D)$. We are interested in the derivative of $\phi$ with respect to $t$ at $t=0$. We denote $t$-derivation by a dot, $w$-derivation by a prime and p-th order $w$ derivatives by ${}^{(p)}$. Now since $f^{(p)}(w)\equiv 0$ for $p\geq 2$ and $f(w)'\equiv 1$, we get for $p\geq 2$ and $q \geq 1$
\begin{eqnarray*}
   \frac{d}{dt}\left(\frac{f(w,t)^{(p)}}{(f(w,t)')^q}\right)_{|t=0} = \frac{\dot f^{(p)} 
   (f')^q - q (f')^{q-1} \dot f' f^{(p)}}{(f')^{2q}}_{|t=0}
   =\dot f^{(p)}(w)\;.
\end{eqnarray*}   
while when the numerator contains products of derivatives of order $\geq 2$,
\benn
   \frac{d}{dt}& \left(\frac{f(w,t)^{(k)}f(w,t)^{(l)}}{(f(w,t)')^q}\right)_{|t=0} \\&= \frac{\left(\dot f^{(k)}f^{(l)} + f^{(k)}\dot f^{(l)} \right) (f')^q - q (f')^{q-1} \dot f' f^{(k)}f^{(l)}}{(f')^{2q}}_{|t=0}\equiv0\;.
\eenn 
Hence only the monomial terms in the numerator survive, and the $t$-derivative of such a monomial at $t=0$ yields
$$ \frac{d}{dt} \phi(w,t)_{|t=0} = \sum a_{k,l} \dot f^{(k)}(w)\;.$$
Now since $f(w,t)$ is holomorphic in $t$, we can expand it as we did in Section \ref{Bersemb},
$$ f(w,t) = f(w) + t f_1(w) + O(t^2)\;,$$
from which of course follows that $\dot f(w,0) = f_1(w)$. We have an explicit representation for the term $f_1$ (see equation \eqref{1stOrderBeltSol}) given by
\begin{equation}\label{EffOne}f_1(w) = -\frac{w(w-1)}{\pi}\int_{\bb C}\frac{\nu(\eta)}{\eta(\eta-1)(\eta - w)}d^2\eta\;.\end{equation}
The integral converges absolutely since the modulus of the integrand is of order $\cal O(|\eta|^{-3})$ for $|\eta| \rightarrow \infty$, hence the $w$-derivation can be moved inside the integral. The part depending on $w$ can be rewritten conveniently as
\begin{equation} \label{OtherNorm} \frac{w(w-1)}{\eta(\eta - 1)(\eta - w)} = \frac{1}{\eta - w} - \frac{w}{\eta - 1} + \frac{w-1}{\eta}\;. \end{equation}
Since the last two terms are linear in $w$, they don't contribute to derivatives of order higher than one, and we can read off the $k$-th $w$-derivative of $\dot f$,
$$ \dot f^{(k)}(w) = \frac {(-1)^{k}k!}{\pi} \int_{\bb D^c} \frac{\nu(\eta)}{(\eta - w)^{k+1}}d^2\eta\;,$$
which then implies that
$$  \frac{d}{dt} \phi(z,t)_{|t=0} = \sum a_{k,l} \frac {(-1)^{k}k!}{\pi} \int_{\bb D^c} \frac{\nu(\eta)}{(z-\eta)^{k+1}}d^2\eta\;,$$
which, of course, is exactly $D_0\tilde \beta^Q(\nu)$, since we know the Gateaux derivative exists. 
\end{proof}
\section{Higher Schwarzian Derivatives and Higher Bers Maps} \label{HBM}
We remarked above that there do exist operators satisfying the prerequisites of our main theorem \ref{higherBersholomorphic}. In this section we introduce two series of such operators in particular, which we call the \emph{A and B series of higher Schwarzians} (Def.~ \ref{HSDDef}), as well as their induced holomorphic mappings, which we call \emph{higher Bers maps} because of their close analogy to the Bers embedding. These operators are good novel examples of non-homogeneous operators that nevertheless induce mappings of Teichm\"uller space. We go into this point more precisely in Section \ref{OtherWork}, where we also review results known on homogeneous operators.
\subsection{Higher Schwarzian Derivatives}\label{HSD}
There are several generalizations of Schwarzian derivatives defined
in the literature. We will consider two particular series of such generalizations, which we call the $A$ and $B$ series. The former are quite recent and were introduced by Eric Schippers in \cite{DTfhoSDoUF}, while the latter have been known for longer time and can be found, for instance, in \cite{CEfNPotDoUF} and \cite{MIOoRS}.
\begin{dfn}\label{HSDDef}
   For any interger $n \geq 3$, the $A$ and $B$ series of \underline{higher} \underline{Schwarzians} $\sigma^\bullet_{n}: \cal M_{\mathrm{li}}( D)\rightarrow        \cal O( D)$ are defined by
   \begin{equation}\begin{aligned}\label{DefSchwarz} 
   				\sigma^A_3[f] &:= S_f\;, \qquad \sigma^A_{n+1}[f] := \sigma^A_n[f]' - (n-1) \frac{f''}{f'}\sigma^A_n[f]\\
   				\sigma^B_n[f] &:= -2(f')^{\frac n2 - 1} \frac{d^{n-1}}{dz^{n-1}}\left((f')^{1-\frac n2}\right)\;,
   				\end{aligned}\end{equation}
   where the same branch of the square root of $f$ is assumed in both appearances in the definition of $\sigma^B_n[f]$ for odd $n$.
\end{dfn}
They are indeed well defined on the space of meromorphic locally injective functions, since one easily can convince oneself that only powers of the first derivative of $f$ appear in the denominator of the expression for $\sigma^\bullet_{n}[f]$ (see also Lemma \ref{StructHSD} below). In general, if we refer to an operator of either series, we will write $\sigma_n^\bullet$. Many statements can be obtained for both types of operators simultaneously. However, to do so we introduced slightly different conventions for the $\sigma^B_n$ than the ones in \cite{CEfNPotDoUF} and \cite{MIOoRS}, where the operators are denoted by $S_n$. More precisely, they are related by $\sigma^B_n = -2 S_{n-1}$. Let us write down the first few operators of both series explicitly. For $\sigma^A$, one obtains via the recursion formula 
\benn
&\sigma^A_4[f] = \frac{f''''}{f'}-6\frac{f'''f''}{(f')^2}+6\left(\frac{f''}{f'}\right)^3\\
& \sigma^A_5[f] = \frac{f'''''}{f'}-10\frac{f''''f''}{(f')^2}-6\left(\frac{f'''}{f'}\right)^2+48\frac{f'''(f'')^2}{(f')^3}-36\left(\frac{f''}{f'}\right)^4\;.
\eenn
We warn the reader that there are two typos in the expression for $\sigma^A_5[f]$ in the original paper \cite{DTfhoSDoUF}. The first operators of the $B$-series are given by
\benn
	&\sigma^B_3[f] = \frac{f'''}{f'}-\frac 32 \left(\frac{f''}{f'}\right)^2\\
  &\sigma^B_4[f] = 2 \frac{f^{''''}}{f'} + 12 \frac{f''' f''}{(f')^2} + 12 \frac{(f'')^2}{(f')^3}\\
  &\sigma^B_5[f] = 3 \frac{f^{'''''}}{f'} - \frac{15}{2} \frac{\left((f''')^2 + 4 f''''f''\right)}{(f')^2}
  + \frac{315}4 \frac{f'''(f'')^2}{(f')^3} - \frac{945}{8} \left(\frac{f''}{f'}\right)^4\;.
\eenn
In particular, $\sigma^\bullet_3[f] = S_f$, so calling them higher Schwarzians is justified.\\ \\
For the derivative of the higher Bers maps, which are holomorphic maps of Teichm\"uller spaces constructed with the help of the higher Schwarzians and will be introduced in a moment, we need the following structural statement about the higher Schwarzians.
\begin{lemma} \label{StructHSD}The expressions $\sigma^\bullet_n[f]$ are polynomials in  $f'',\ldots,f^{(n)}$ and $(f')^{-1}$, with integer coefficients for the A series (except for $\sigma_3^A[f]$, of course) and rational coefficients for the B series. The only term of $\sigma^\bullet_n[f]$ where the numerator is a monomial of degree one in $f'',\ldots,f^{(n)}$ is $c^\bullet(n) \cdot f^{(n)}/f'$ where $c^A(n) = 1$ and $c^B(n) = n-2$.
\end{lemma}
\begin{proof} Let us first consider $\sigma^A_n[f]$. The lemma is certainly true for $n\leq 5$ by the explicit formulas above. Now $\sigma^A_{n+1}[f]$ is a sum of the derivative of $\sigma_n^A[f]$ and the product $f''/f' \cdot \sigma^A_n$. By induction, the latter term being a product of two polynomials in $f'',\ldots,f^{(n)}$ and $(f')^{-1}$ is again a polynomial in these variables, and since each of the polynomials has no constant term, this product cannot contain a monomial of degree one. Also, the derivative of $\sigma_n^A[f]$ is of this structure by the quotient rule of differentiation and a monomial of degree one can only be obtained by differentiating the term $ f^{(n)}/f'$. The monomial obtained in this way is $f^{(n+1)}/f'$. Since $\sigma^A_4[f]$ has integer coefficients only, so do all $\sigma_n^A[f]$ with $n \geq 4$ because the recursive relation does not produce rationals.\\ \\
Let us now consider the $\sigma^B_n$. Let us compute the first derivatives in the definition of $\sigma^B_n[f]$,
\benn
 \frac{d^{n-1}}{dz^{n-1}}& \left((f')^{1- \frac n2}\right)= \left(1 - \frac n2\right) \frac{d^{n-2}}{dz^{n-2}}\left((f')^{1- \frac n2 -1} f''\right)\\
&=\left( 1- \frac n2\right)\frac{d^{n-3}}{dz^{n-3}}\left((1-{\textstyle \frac n2} - 1) (f')^{1 - \frac n2 - 2} (f'')^2 + (f')^{-\frac n2} f'''\right)\;,
\eenn
from which we see what happens in general. Namely, the result of taking all derivatives will be a sum of terms, each of which is a product of $f'', \ldots, f^{(n)}$ and $(f')^{-\frac n2 - k}$ with $0 \leq k \leq n-2$. Hence, after multiplying through with $(f')^{\frac n2 -1}$, the summands are products of $f'', \ldots, f^{(n)}$ and $(f')^{-(k+1)}$. Moreover, the only term with numerator a monomial of degree 1 in $f'', \ldots, f^{(n)}$ is the second summand obtained by the product rule of differentiation applied to the term $(f')^{-\frac n2} f^{(k)}$. The resulting term after $n-1$ derivatives in $\sigma_n^B[f]$ is then 
$$(-2) \cdot (1- {\textstyle \frac n2} ) (f')^{\frac n2 -1} (f')^{-\frac n2} f^{(n)} = (n-2) \frac{f^{(n)}}{f'}\;.$$
\end{proof}
The coefficients in both series are indeed rational and not integer. Our examples above already show this for the $B$ series. The first non-integer coefficient in the $A$ series appears in $\sigma^A_6[f]$.\\ \\
There is a very useful reformulation of the recursion relation for $\sigma_n^A$, which we want to consider next. It will be used several times later on.
\begin{lemma}\label{AltDefHS}
   The defining relation (\ref{DefSchwarz}) for the operators in the $A$ series of higher Schwarz\-ians can be rewritten as follows
   \begin{eqnarray} \label{AltDefSchw} \frac{\sigma^A_{n+1}[f]}{(f')^{n-1}} = \left( \frac{\sigma^A_{n}[f]}{(f')^{n-1}}\right)'\;.\end{eqnarray}
   Moreover, if $f$ is schlicht, this is equivalent to 
   \begin{eqnarray} \label{SchwInv} \sigma^A_n[f]=\left(\frac{d^{n-3}}{dz^{n-3}} S_{f^{-1}} \right)\circ f \cdot (f')^{n-1}\;.\end{eqnarray} 
\end{lemma}
\begin{proof}. By applying the quotient rule to (\ref{AltDefSchw}) we immediately get (\ref{DefSchwarz}) and thereby the first statement is proven. The second statement follows from writing the Schwarzian of the inverse function in terms of the Schwarzian of the function itself by applying the chain rule for Schwarzians to $f^{-1}\circ f$, 
$$ -S_{f^{-1}} \circ f = S_f \cdot (f')^{-2}\;.$$ 
If we take the derivative of this equation, we get
$$ -\left(\frac{d}{dz} S_{f^{-1}} \right)\circ f \cdot f' = \left(\frac{S_f}{(f')^2}\right)' = \frac{\sigma^A_4[f]}{(f')^2}\;,$$
and this proves (\ref{SchwInv}) for $n=4$. Inductively, we assume (\ref{SchwInv}) to be true for $n$, divide it by $(f')^{n-1}$ and take the derivative of it,
$$ \left(\frac{\sigma^A_n[f]}{(f')^{n-1}}\right)' = \left(\frac{d^{n+1-3}}{dz^{n+1-3}} S_{f^{-1}}\right) \circ f \cdot f'\;.$$
If we now multiply through with $(f')^{n-1}$, the left hand side equals $\sigma^A_{n+1}[f]$ by equation (\ref{AltDefSchw}) whereas the right hand side is the same as the right hand side of (\ref{SchwInv}) for the value $n+1$. 
\end{proof}
Neither series of higher Schwarzian derivatives has a nice chain rule, i.e., a closed formula for the value of the operator applied to a composition of functions. However, they do behave nicely when precomposed with M\"obius transformations. Observe that this is precisely the required formula in the main Theorem \ref{higherBersholomorphic}
\begin{lemma}\label{HSDMoebius}
The higher-order Schwarzian derivatives behave in the following way under precomposition with a M\"obius transformation $g$,
\begin{eqnarray}
    \sigma^\bullet_{n+1}[f \circ g] = (\sigma^\bullet_{n+1}[f] \circ g) (g')^{n}\;. \qquad
\end{eqnarray}
\end{lemma}
\begin{proof}
The formula is proved by induction for the $A$ series. By \eqref{trafoSchw} we know it is true for $\sigma^A_3[f]=S_f$. For the induction step we compute,
\benn
    \sigma^A_{n+1}[f\circ g] &= \left(\sigma^A_{n}[f]\circ g \cdot (g')^{n-1}\right)' - (n-1)\frac{(f\circ g)''}{(f\circ g)'} \sigma^A_{n}[f\circ g]\\
    &= \left(\sigma^A_{n}[f]'\circ g\right) \cdot (g')^{n} + (n-1)\sigma^A_{n}[f]\circ g \cdot (g')^{n-2}\cdot g'' \\
    &\quad - (n-1) \frac{(f\circ g)''}{(f\circ g)'} (g')^{n-1}\sigma^A_{n}[f]\circ g \;.
\eenn
Comparing this with the desired result written out explicitly,
$$ (\sigma^A_{n+1}[f]\circ g)\cdot(g')^{n} = \left(\sigma^A_{n}[f]'\circ g - (n-1)\frac{f'' \circ g}{f'\circ g}\sigma^A_{n}[f]\circ g\right)\cdot (g')^{n}\;,$$
we find that the first term matches up fine while the second matches up, iff
$$ g'' - \frac{(f\circ g)''}{(f\circ g)'} g' = -(g')^2 \frac{f'' \circ g}{f'\circ g}\;. $$
But this is easily seen to be true, since
$$\frac{(f\circ g)''}{(f\circ g)'} = \frac{f''\circ g \cdot (g')^2 + f' \circ g \cdot g''}{f'\circ g \cdot g'}\;.$$
This concludes the proof for the $A$ series. As for the proof of the transformation behaviour of the $B$ series we reproduce the elegant proof of this fact contained in \cite{CEfNPotDoUF}. This proof uses a well-known lemma due to Bol \cite{IlD}.
\begin{lemma}
   Let $f_i \in \cal O(D)$ be related via $f_2 = (f_1 \circ g) (g')^{1-\frac n2}$ for some $g \in \mathrm{PSL}(2, \bb C)$. Then their $(n-1)$st derivatives are related by
   $$ f_2^{(n-1)} = (f_1^{(n-1)} \circ g) (g')^{\frac n2}\;.$$
\end{lemma}
We apply this lemma to the functions
$$ f_1 = (f')^{\frac n2 -1}\;, \qquad f_2 = \left((f \circ g)'\right)^{1-\frac n2} = \left(f'\circ g \cdot g'\right)^{1-\frac n2}\;, $$ where $g$ is any M\"obius transformation. Observe that their quotient is precisely $\sigma^B_n[f\circ g]$, which then immediately yields the desired transformation behaviour,
\begin{eqnarray*} \sigma^B_n[f\circ g] &=& \left((f \circ g)'\right)^{\frac n2-1} \frac{d^{n-1}}{dz^{n-1}}\left( \left(f \circ g)'\right)^{\frac n2-1}\right)\\ &=& \frac{f_2^{(n-1)}}{f_2} = \frac{f_1^{(n-1)} \circ g}{f_1 \circ g} (g')^{n-1} = (\sigma^B_n[f] \circ g) (g')^{n-1}\;.\end{eqnarray*}
\end{proof}
But what happens under postcompositions with M\"obius transformations? Neither series of operators has an invariance property with respect to postcompositions, but for certain M\"obius transformations, the A series behaves invariantly. Namely, observe that the inductive formula \eqref{DefSchwarz} can be written with the help of the \emph{pre-Schwarzian}, i.e., the operator given by $PS[f]:=f''/f'$,  
$$\sigma^A_{n+1}[f] := \sigma^A_n[f]' - (n-1) PS[f] \sigma^A_n[f]\;.$$
Hence the $\sigma^A_n$ will be invariant under the postcomposition with those maps that leave \emph{both} the Schwarzian and the pre-Schwarzian invariant. These are necessarily M\"obius transformations because of the required invariance of the Schwarzian; a simple computation further shows that the pre-Schwarzian is only left invariant by \emph{affine transformations},
$$ PS[M] = \frac{M''}{M'} = (cz +d)^2 \cdot (-2) \frac{c}{(cz+d)^3} = \frac{-2c}{cz +d}\;,$$
which vanishes only for $c=0$. This proves the following lemma.
\begin{lemma}
    Let $M$ be an affine transformation, i.e., $M(z) = az+b$. Then $\sigma^A_n[M\circ f] = \sigma^A_n[f]$.
\end{lemma}
Both series of higher Schwarzians map meromorphic locally injective functions to holomorphic functions. We equipped the target space $\cal O( D)$ with various norms which induce Banach space structures in Section \ref{BSoAF}. There are deep relations between the higher Schwarzians and the hyperbolic sup-norms, which yield analogous versions of the Kraus-Nehari Theorem. These results are a further motivation for the study of higher Bers maps.
\begin{prop}[\cite{DTfhoSDoUF}]\label{higherbound}
If $f$ is schlicht, then
$$ \| \sigma^A_{n}[f]\|_{B_{n-1}(\bb D)} \leq 4^{n-3}(n-2)!6\;,$$
and this bound is sharp.
\end{prop}
We will not reproduce the proof here, because it is quite involved and instead refer the interested reader to the original paper. The author of \cite{DTfhoSDoUF}, Eric Schippers, has communicated to me that this estimate is already implicitly contained in \cite{TNEPotK-F}, of which he was not aware at the time of writing \cite{DTfhoSDoUF}.\\ \\
The analogous result for the $B$-series follows; neither will we reproduce this proof here.
\begin{prop}[\cite{CEfNPotDoUF}, Thm.~ 2]\label{higherbound2}
   If $f$ is schlicht, then
   $$ \| \sigma^B_{n}[f]\|_{B_{n-1}(\bb D)} \leq 2 (n-2)\cdot n \cdot (n+2)\cdot \ldots \cdot (3n-6)\;,$$
   and this bound is sharp.
\end{prop}
Both formulas of course reproduce the constant $6$ for $n=3$ and, in both cases, the functions which realize the bounds are related to the Koebe function. 
\subsection{Higher Bers Maps: Definition and Infinitesimal Properties}\label{HBMInfinitesimal}
We now begin to tie the higher Schwarzians to the geometry of Teichm\"uller space in precise analogy to the Bers embedding.
\begin{dfn}\label{HBMDef}
   The \underline{higher Bers maps} $\beta^\bullet_n$ are given by,
   $$ \beta^\bullet_n : \cal T_B(G) \rightarrow B_{n-1}(\bb D)\;, \qquad \beta^\bullet_n([\mu]) = \sigma^\bullet_n(\pi_{BS}([\mu]))\;.$$
   Moreover, the lift of $\beta^\bullet_n$ to $L^\infty(\bb D^c)$ will be denoted by $\tilde \beta^\bullet_n$ and the image of $\cal T_B(G)$ in $B_{n-1}(\bb D)$ will be denoted by $T^\bullet_n(G)$.  
\end{dfn}
The main theorem \ref{higherBersholomorphic} applies to these maps because of Lemma \ref{HSDMoebius} and hence we get the following corollary.
\begin{cor}
   The higher Bers maps $\beta^\bullet_n$ are holomorphic maps.
\end{cor}
The remainder of this section focusses on the infinitesimal behaviour of the higher Bers maps. Many results from the classical case of the Bers embedding generalize in some sense for the \emph{differential at the origin}, which we have obtained in general in Theorem \ref{ThmDifferential}. We determine its surjectivity (Thm.~ \ref{HBDiffSurj}), explicitly describe its kernel (Thm.~ \ref{KernelHBM}), and obtain several other results along the way.\\ \\
Indeed, almost all crucial properties of the Bers embedding are deduced from the differential at the origin, since it can be related to the differential at an arbitrary point of $\cal T_B(G)$ by taking `the derivative of the chain rule'. The lack of a good chain rule for higher Schwarzians therefore makes it difficult to obtain infinitesimal statements at other points. We will comment on generalizations and further ideas to circumvent this problem in the next section.\\ \\
Let us now specialize Theorem \ref{ThmDifferential} to the higher Bers maps.
\begin{prop}
   The differentials of the higher Bers maps at the origin are given by 
   $$ D_0\tilde \beta^\bullet_n: L^\infty(\bb D^c) \rightarrow B_{n-1}(\bb D)\;, \quad \nu \mapsto \frac{(-1)^{n} n! c^\bullet(n)}{\pi} \int_{\bb D^c}      \frac{\nu(\eta)}{(z-\eta)^{n+1}}d^2\eta\;.$$
   The operator norm of $D_0\tilde \beta^\bullet_n$ is bounded by $\frac{2\cdot 4^{n-1}n!c^\bullet(n)}{n-1}$
\end{prop}
\begin{proof} By Lemma \ref{StructHSD}, the monomial part of $\sigma^\bullet_n[f]$ is given by $c^\bullet(n)\frac{f^{(n)}}{f'}$, and this proves the formula for the differential together with Theorem \ref{ThmDifferential}. The norm estimate follows easily from the integral estimate,
\begin{eqnarray*}
   \int_{\bb D^c} \frac{1}{|\eta - w|^{n+1}}d^2\eta&\leq&\int_{|\eta - w|\geq 1-|w|}\frac{1}{|\eta - w|^{n+1}}d^2\eta\\
   &=& 2\pi \int^{1-|w|}_\infty \frac{1}{r^{n+1}}r dr = \frac{2\pi}{(n-1)(1-|w|)^{n-1}}\;. 
\end{eqnarray*}
together with the observation that $(1-|w|) = \textrm{dist}(w, \partial \bb D)$. Therefore, by the fundamental inequality \eqref{FEoPD} for the Poincar\'e density, we get
\benn
| D_0\tilde \beta_n^\bullet(\nu)(w) \lambda_{\bb D}^{1-n}(w)| &\leq \|\nu\|_\infty \frac{2\cdot n!\lambda_{\bb D}^{1-n}(w)c^\bullet(n) }{(n-1)(1-|w|)^{n-1}}\\&\leq \frac{2\cdot 4^{n-1}n!c^\bullet(n)}{n-1}\|\nu\|_\infty\;.
\eenn
\end{proof}
Since the differentials at the origin of the two series of higher Bers maps are proportional, we begin to study this important linear operator in more detail. In fact, Bers' proof that $\beta^\bullet_3$ is an embedding consists of two parts:
First of all, Bers established the necessary properties of the differential at the origin, and as second step he showed that the differential at an arbitrary $\mu \in \mathrm{Belt}(\bb D^c)$ is related to the differential at the origin by composition with isomorphisms, which are obtained from considering the chain rule for Schwarzians and the translation maps in Belt$(\bb D^c)$, very similarly to a computation we do later in \eqref{Isos}.\\

The first part, studying the differntial at the origin, however, is more involved. Bers accomplished the proof of surjectivity (amongst many other results) in the beautiful paper \cite{An-sIEwAtQM}. At the heart of the surjectivity proof lies the following reproducing formula.
\begin{thm}[Bers]\label{RepForm}
   Let $D_1$ be a quasidisc with $\infty \in \partial D_1$, $h:D_1 \rightarrow D_2$ a uniform Lipschitz reflection      across $\partial D_1$ and $q\geq 2$ an integer. Then the following reproducing formula holds:
   \begin{eqnarray} \label{RepFormula}
    \phi(z) = \int_{D_1}\frac{\nu_\phi^q(\eta)}{(z-\eta)^{2q}} d^2\eta\;, \qquad \forall \phi \in B_q(D_2)\;,              \end{eqnarray}
   where $\nu^q_\phi$ is given by
   $$ \nu_\phi^q(z):= -\frac{2q-1}{\pi}(\phi \circ h)(z)\cdot \partial_{\bar z} h(z)\cdot\big(z - h(z)\big)^{2q-2}\;.$$
\end{thm}
What one should notice here is that $\nu^q$ is \emph{not} a continuous linear operator from $B_q(D_2)$ to $L^\infty(D_1)$, because the norm of $\nu^q_\phi$ cannot be uniformly estimated by the $B_q$-norm of $\phi$. The best one achieves is
$$|\nu^q_\phi(z)| \leq C \| \phi \|_{B_q(D_2)} \lambda_{D_1}^{q-2}(z)\;.$$
The unsatisfactory point is the appearance of the unbounded quantity $\lambda_{D_1}$, which enters due to the term $|h(z) - z|$. The $B_q$-norm is only capable of absorbing a power $q$ of the Poincar\'e density, and hence a power of $q-2$ cannot be taken care of. For convenience, let us introduce a slightly different quantity,
$$ \mu^q_\phi(z) := \nu_\phi^{\frac{q+2}{2}} (z)= C_q (\phi \circ h)(z)\cdot \partial_{\bar z} h(z)\cdot\big(z - h(z)\big)^{q}\;.$$
in terms of which the reproducing formula reads
\begin{eqnarray} \label{Q-rep} \phi(z) = \int_{D_1}\frac{\mu_\phi^q(\eta)}{(z-\eta)^{q+2}} d^2\eta\;, \qquad \forall \phi \in B_{\frac{q+2}{2}}(D_2)\;. \end{eqnarray}
By rewriting, we have obtained a reproducing formula with different domain (i.e., $B_{\frac{q+2}{2}}$ instead of $B_q$), but with the structure we need. In Proposition \ref{Correct!} below, we will show that this formula is verbatim valid on $B_q$. Observe, however, that for $q=2$ the two agree, i.e., $\mu_\phi^2 = \nu_\phi^2$, so if one looks at the case of the Bers embedding this difference is not seen. \\

The reproducing formula proves the surjectivity of the differential of the Bers embedding (say, modeled on the upper half-plane instead of the disc for simplicity) directly:
$$ D_0\tilde \beta_3 (\nu_\phi^2)(z) = -\frac{6}{\pi} \int \frac{\nu_\phi^2(w)}{(z-w)^4} d^2w = -\frac{6}{\pi} \phi(z)\;,$$
so in other words $\phi \mapsto -\frac{\pi}{6} \nu^2_\phi$ is a section of the map $D_0\tilde \beta_3$.\\

To see that the original reproducing formula is not sufficient for the differential of the $n$-th higher Bers map, observe that this contains the term
$(z-w)^{n+1}$. The appropriate $q(n)$ in the reproducing formula \ref{Q-rep} is given by
$$n+1 \shb q+2 \qquad \Longrightarrow \qquad q(n) = \frac {n-1}{2}\;.$$
However, the formula with this $q(n)$ is only valid on $B_{\frac{q(n)+2}{2}}$, whereas we would need it on $B_{n-1}$. This is unfortunately not the case, however, since 
$$ \frac{q(n)+2}{2} = \frac{n+1}{2} \neq n-1\qquad \forall \: n\geq 4\;.$$
So \eqref{Q-rep} only proves the surjectivity of the differential for $n=3$. But fortunately the argument of Bers' original proof goes through with some modification for the setting in which we need it and a version of the same fomula holds, which we give now.
\begin{prop} \label{Correct!}
   Let $D_1,D_2, h$ and $q$ be as in Theorem \ref{RepForm}. Then 
   $$ \phi(z) = \int_{D_1}\frac{\mu_\phi^q(\eta)}{(z-\eta)^{q+2}} d^2\eta\;, \qquad \forall \phi \in B_{q}(D_2)\;.$$
\end{prop}
\begin{proof}
We won't reproduce the whole proof here but rather sketch it. One starts by proving it for holomorphically extendable functions
$$ \psi \in \tilde A_q^1(D_2):=\left \{ f \in A_q^1(D'_f) \;\mathrm{for \; some}\: D'_f \supset \overline {D_2}\right\}\;.$$
Holomorphicity at $\infty$ implies $f \in \cal O(|z|^{-2})$ which is satisfied anyway, since $f \in A_q^1(D'_f)$, which implies \mbox{$f \in O(|z|^{-2q})$} for $z \rightarrow \infty$. In the same way, the existence of the $j$-th derivative at infinity requires $f \in \cal O(|z|^{-(j+1)})$. Hence for $\psi \in \tilde A_q^1(D_2)$, we know there is a function $F_j \in \cal O(D'_\psi)$ such that
$$ F^{(j+1)}_j (z)= \psi(z) \qquad \forall \: \bb N \ni j \leq 2q-1\;.$$
Then define the function
$$ G_j(z):=\left\{ \begin{array}{ll} F_j(z) & z \in \overline{D_2} \\ \sum_{k=0}^{j}\frac{1}{k!}(z - h(z))^k F^{(k)}_j(h(z))& z \in D_1 \end{array}\right.$$
If we compute $\bar \partial G_j$, which of course vanishes on $D_2$, some nice cancellations occur, since the derivative of a summand in $G_j$ is given by
\benn
\bar \partial & \left((z - h(z))^k F^{(k)}_j(h(z))\right)= \\
 &-\bar \partial h(z) \cdot k(z-h(z))^{k-1} \cdot F^{(k)}_j(h(z)) + (z-h(z))^{k} F_j^{(k+1)}(h(z)) \bar \partial h(z)\;,
\eenn
where if summed up the first part of the $k$-th summand cancels the second part of the $(k-1)$-st summand, so that only the second part of the derivative of the last summand remains, namely
$$ \bar \partial G_j(z) = \frac{1}{j!}(z-h(z))^j \bar \partial h(z) F^{(j+1)}_j(h(z)) = -\frac{\pi}{(j+1)!} \mu^j_\psi(z) \qquad \forall \: z \in D_1\;.$$
On the other hand, $G_j$ has the required regularity for Green's formula to hold,
$$ G_j(z) = -\frac{1}{\pi} \int_{\bb D_R} \frac{\bar \partial G_j(w)d^2w}{w-z} + \frac{1}{2\pi i}\int_{\partial \bb D_R} \frac{G_j(w) dw}{w-z}\;,$$
which we differentiate $j+1$ times with respect to $z$. By construction, the left-hand side is $\psi$ whereas on the right-hand side, differentiation produces a factor of $(j+1)!$ so we get
$$ \frac{d^{j+1}}{dz^{j+1}}G_j(z) = \psi(z)= \int_{\bb D_R} \frac{\mu^j_\psi (w) d^2w}{(w-z)^{j+2}} + \frac{(j+1)!}{2\pi i}\int_{\partial \bb D_R} \frac{G(w) dw}{(w-z)^{j+2}}\;,$$
In the first term we can write the integral over $D_R := \bb D_R \cap D_1$ since else the integrand is zero. If we now take the limit $R\rightarrow \infty$, the second term vanishes and the first term becomes the desired reproducing formula. Finally the proof concludes by an approximation argument of functions in $B_q$ by functions in $\tilde A_q^1$ which is exactly the same as in \cite{An-sIEwAtQM}.
\end{proof}
With this modified version of the reproducing formula, we can proceed similarily to Bers' original proof of the surjectivity of the differential of the Bers mapping. We want to remark that the operator which is given by the derivative of the higher Bers maps already appears in \cite{UFXXX} and there also the surjectivity is established.  
\begin{thm} \label{HBDiffSurj}
    The differentials of the higher Bers maps at the origin, $D_0\beta^\bullet_n$, are surjective operators.
\end{thm}
\begin{proof}
Of course $D_0\beta^\bullet_n$ will be surjective iff $D_0\tilde \beta^\bullet_n$ is. We observe that formally the reproducing formula already does the job: Set $q=n-1$, then
$$ \phi(z) = \int_{D_1} \frac{\mu^{n-1}_\phi(w)}{(z-w)^{n+1}} d^2w  =(-1)^{n}  \frac{n! c^\bullet(n)}{\pi} D_0\tilde \beta^\bullet_n(\mu^{n-1}_\phi)\;.$$  
However, recall that the formula required $\infty \in \partial D_2$, e.g., $D_2=\bb H^c$. In order to apply the formula on $B_{n-1}(\bb D)$, let $g:\bb H \rightarrow \bb D$ be a M\"obius transformation. It acts on $\cal S(\bb D)$ by pull-back, and hence also on $\cal T_S(\bb 1)$. Now
$$ \hat \mu( g^*f) = \hat \mu (f) \circ g \cdot \frac {\bar g'}{g'} = g^*_{(-1,1)}\hat \mu(f)\;,$$
and since
$$ \sigma^\bullet_n[g^*f] = \sigma^\bullet_n[f\circ g] = (\sigma^\bullet_n[f] \circ g) (g')^{n-1} \qquad \forall \; g\in \textrm{M\"ob}(\cinf)\;,$$
we have the following commutative diagram:
\begin{diagram}
  \cal T^{\bb H}_S(\bb 1)  &   \rTo^{\hat \mu}   &L^\infty(\bb H^c)&\rTo^{\tilde \beta^{\bullet, \bb H}_n}&B_{n-1}(\bb H)\\
  \uTo_{g^*}&&\uTo_{g^*_{(-1,1)}}&&\uTo_{g^*_{n-1}}\\
  \cal T^{\bb D}_S(\bb 1)  &   \rTo^{\hat \mu}   &L^{\infty}(\bb D^c)&\rTo^{\tilde \beta^\bullet_n}&B_{n-1}(\bb D)\\
\end{diagram}
For the time being we have attached the superscripts $\bb D$ resp.~ $\bb H$ to distinguish the different spaces resulting from a different model domain. Now, if we look at the proof of Theorem \ref{ThmDifferential}, no use whatsoever was made of the fact (see remark below) that the Beltrami differentials were supported on $\bb D^c$. Hence the differential
$$ D_0\tilde \beta^{\bullet, \bb H}_n: L^\infty(\bb H^c) \rightarrow B_{n-1}(\bb H)\;,$$
is given by 
$$ D_0\tilde \beta^{\bullet, \bb H}_n(\nu) = \frac{(-1)^n n!}{\pi}\int_{\bb H^c}\frac{\nu(w)}{(w-z)^{n+1}}d^2w\;.$$
Now we can utilize the reproducing formula. By rewriting it in terms of the differential we see that it states
$$ \phi = \frac{\pi}{(-1)^n n!} \left(D_0\tilde \beta^{\bullet, \bb H}_n\circ \mu^{n-1}\right)(\phi)\;, \qquad \forall \; \phi \in B_{n-1}(\bb H)\;,$$
which especially implies that $D_0\tilde \beta^{\bullet, \bb H}_n$ is a surjective operator. The proof concludes by relating the two differentials. By taking the derivative in the commutative diagram we obtain
\begin{eqnarray} \label{Isos} D_0\tilde \beta^\bullet_n = (g^*_{n-1})^{-1}\circ D_0\tilde \beta^{\bullet, \bb H}_n\circ g^*_{(-1,1)}\;,\end{eqnarray}
where the operators to the left and right of $D_0\tilde \beta^{\bullet, \bb H}_n$ are isomorphisms. Hence $D_0\tilde \beta^\bullet_n$ is surjective as well. 
\end{proof}
\begin{rmk}
In the proof we have used the explicit expression for $f_1$ given in \eqref{EffOne}. This expression depends on the normalization, i.e., on the fact that $f(z,0)=z$, or equivalently, on the fact that $f_0=z$. However, the general formula for $f_1$ without any assumption on the normalization is obtained by inserting coefficients $A,B$ in front of the two last summands (see \eqref{OtherNorm}) on the right hand side of
$$ \frac{w(w-1)}{\eta(\eta-1)(\eta - w)} = \frac{1}{\eta - w} - \frac{w}{\eta - 1} + \frac{w-1}{\eta}\;.$$
These, however, do not enter into the differntial $D_0\beta^\bullet_n$ because this is always at least the third derivative of $f_1$. By the same argument, the same structural term for $D_0\beta^{\bullet, \bb H}_n$ is justified. Indeed, the differential of the higher Bers maps is independent of the chosen normalization, in contrast to the maps themselves.
\end{rmk}
The next step in the infinitesimal study of the higher Bers maps is to identify the kernel of the differential. For this we observe that we can rewrite
\begin{eqnarray*}
D_0\tilde \beta^\bullet_n(\nu)(z)\sim \int_{\bb D^c} \frac{\nu(w)}{(w-z)^{n+1}} d^2w 
              = \int_{\bb D^c}\frac{ \nu(w) \lambda_{\bb D^c}^{2q-2}(w)}{(w-z)^{n+1}} \lambda^{2-2q}_{\bb D^c}(w) d^2w\;,
\end{eqnarray*}
so if we define the functions
$$ \omega_z^{l} (w) := \frac{1}{(w-z)^{l}} \qquad \forall z \in \bb D\;,$$
we \emph{formally} get the identity
\begin{eqnarray} \label{FormalProduct} D_0\tilde \beta_n^\bullet(\nu)(z)&\sim& \left \langle \omega_z^{l}, \; \overline{\nu \omega_z^{l'}}\lambda^{2q-2}_{\bb D^c}\right\rangle^{\bb 1}_q\;, \qquad \mathrm{for}\quad l+l' = n+1\;. \end{eqnarray}
We say formally, because in order for the Weil-Petersson pairing to be defined and finite, we need the pair of paired functions to satisfy
$$ \Big( \omega_z^{l}, \; \overline{\nu \omega_z^{l'}}\lambda^{2q-2}_{\bb D^c}\Big) \in L^p_q(\bb D^c) \times L^{p'}_q(\bb D^c)\;, \qquad \mathrm{with} \quad \frac 1p + \frac {1}{p'} = 1\;.$$
The reason we split up the terms in this way in \eqref{FormalProduct} is that we want to keep one of the factors \emph{holomorphic}. We clarify the possibilities in the following technical lemma.
\begin{lemma}
   The following statements hold for fixed $q\geq 2$:
   \begin{eqnarray*}
   		\|\omega_z^{l}\|_{A^p_q(\bb D^c)}  &<& \infty \; \Longleftrightarrow \; 
   		\begin{cases}  1 \leq p < \infty \; \mathrm{and} \; l>q \\ \mathrm{or} \; p= \infty \; \mathrm{and}\;                             l\geq q \end{cases} \\
   		\|\overline{\nu \omega_z^{l'}}\lambda^{2q-2}_{\bb D^c}\|_{L^{p'}_q(\bb D^c)} &<& \infty \; \Longleftrightarrow \;
   		\begin{cases} 2 \leq p' < \infty \; \mathrm{and} \; l'> 2 + (p'-2)q \\ \mathrm{or} \; p'=\infty        \;                    \mathrm{and} \; l' \geq q-2 \;.\end{cases}
   \end{eqnarray*}
\end{lemma}
\begin{proof} We start by proving the first statement. This is a simple estimate,
\begin{eqnarray*}
	\|\omega_z^{l}\|^p_{A^p_q(\bb D^c)} &=&\int_{\bb D^c} |\omega_z(w)|^{lp}\lambda_{\bb D^c}^{2-pq}(w) d^2w \\
	&\leq& C \int_{\bb D^c} \frac{1}{|w-z|^{(l-q)p + 2}|w-z|^{pq-2}\lambda_{\bb D^c}^{pq-2}(w)}d^2w \\ &\leq& 2 \pi C' \int_{R}^{\infty} \frac{1}{r^{(l-q)p+1}}dr
	< \infty \qquad \Leftrightarrow \quad (l-q)p+1>1\;,
\end{eqnarray*}
which is the same as $l>q$. The second step in the estimate follows by estimating $|z-w|\geq \textrm{dist}(z,\partial D)$ and once again using the asymptotic property $\lambda_D(z)\textrm{dist}(z,\partial D) \in O(1)$ g for $z\rightarrow \infty$ (see Eq.~ \eqref{FEoPD}). $R$ is chosen such that $0<R< \textrm{dist}(z,\partial \bb D)$. The case $p = \infty$ follows similarly by the same ingredients, but there we only need $l\geq q$ since $|w-z|^\alpha$ is bounded for $\alpha \leq 0$. For the second statement, a very similar computation yields the requirements $p' \geq 2$ and $l'> 2 + (p'-2)q$, where the first one comes from the fact that the total power of the Poincar\'e density in the integral this time is $2q-2+2-qp' = (2-p')q$, which should be $\leq 0$ for the factor to possibly be compensated as in the case above. So we get $p'\geq 2$. After compensating the powers of $\lambda_{\bb D^c}$, the integrability requirement yields the second inequality. The $p'=\infty$ estimate again uses the boundedness of $|\nu|$ and $|z-w|^{-1}$, the power of which is $l'$, which has to compensate the power $q-2$ of $\lambda_{\bb D^c}$, so $l'\geq q-2$.
\end{proof}
Beware however, as finiteness of the $L^p_q(D)$-norm does not always mean that the function belongs to $L^p_q(D)$: It also has to be of order $O(|z|^{-2q})$ for $|z| \rightarrow \infty$ if $\infty \in D$. This condition is vacuous if $\infty \notin D$, of course, and that is the reason we stated the lemma that way above. For our applications we will need the following special case, namely $l'=0$ and $l=n+1$. But then necessarily $p'=\infty$, hence $p=1$, and then $0=l'\geq q-2$ implies that $q=2$. Then in addition the decay condition at infinity implies the following corollary.
\begin{cor}
    We have for $q\geq 2$,
    $$\omega_z^l \in A^1_q(\bb D^c) \quad \Leftrightarrow \quad l\geq 2q\;, \qquad  \overline{\nu}\lambda^{2q-2}_{\bb D^c} \in L^{\infty}_q(\bb D^c) \quad \Leftrightarrow \quad q=2\;.$$ 
\end{cor}  
So returning to our point in \eqref{FormalProduct} and using the previous corollary, we get the following criterion for a measurable function to be in the kernel of the derivative of the higher Bers mappings:
$$ D\tilde \beta^\bullet_n (\nu) \equiv  0 \qquad \Longleftrightarrow\qquad \left \langle \omega_z^{n+1}, \bar \nu\lambda^{2}_{\bb D^c} \right \rangle^{\bb 1}_2 = 0 \quad \forall \; z \in \bb D\;.$$
We can optimize this further by using the projection operator $\beta_q: L^p_q(D,G)\rightarrow A^p_q(D,G)$, which is symmetric with respect to the Weil-Petersson pairing (see Theorem \ref{last}), in particular implying that
$$ \left \langle \beta_q (f),\nu \right \rangle^G_q =  \left \langle f,\nu \right \rangle^G_q = \left \langle f,\beta_q (\nu) \right \rangle^G_q \quad \mathrm{if} \; f \in A^1_q(D,G),\; \nu \in L^\infty_q(D,G).$$ 
Hence we get the following criterion for the kernel:
\begin{eqnarray} \label{KernelCrit} D_0\tilde \beta^\bullet_n (\nu) \equiv  0 \quad \Longleftrightarrow\quad \left \langle \omega_z^{n+1}, \beta_2\left(\bar \nu\lambda^{2}_{\bb D^c}\right) \right \rangle^{\bb 1}_2 = 0 \quad  \forall \; z \in \bb D\;. \end{eqnarray}
Since both entries are now holomorphic, we can use the fact that the pairing restricted to $A^p_q\times A^{p'}_q$ is non-degenerate (Thm.~ \ref{NondegenerateOnAs}). Moreover, the following map, known as the (generalized) Bers differential,  
$$ \psi_2: L^\infty(\bb D^c) \rightarrow L^\infty_2(\bb D^c)\;, \qquad \psi_q(\nu) =  \bar \nu  \lambda^2_{\bb D^c}\;,$$
is obviously an isometric isomorphism which induces isometric isomorphisms for any Fuchsian group $G$ by restriction,
$$ \psi_q: L^\infty_{(-1,1)} (\bb D^c, G) \rightarrow L^\infty_2(\bb D^c, G)\;.$$
Since $\beta_2$ is a projection, the right entry of the pairing \eqref{KernelCrit} generates all of $B_2(\bb D^c)$, and likewise, it generates all of $B_2(\bb D^c, G)$ iff $\beta_2$ is restricted to $L^\infty_{(-1,1)} (\bb D^c, G)$. Now by Lemma \ref{scalar}, if one of the entries, say $f$, in the Weil-Petersson pairing is $q$-automorphic for the group $G$, then we can rewrite the product using the Poincar\'e series operator,
$$ \langle f,g\rangle^{\bb 1}_q = \langle f, \Theta_q[g] \rangle^{G}_q \;.$$
So if we define the following spaces for integer $l\geq 2q$,
$$ A_q^1(\bb D^c, G) \supset \cal A^l_q(G):=\mathrm{cl}_{A_q^1(\bb D^c, G)}\left(\mathrm{span}_\bb C\left \{ \Theta_q[\omega^l_z], z \in \bb D \right \} \right)\;,$$
and as usual drop the group $G$ from the notation if $G = \bb 1$, then we have proven the following theorem:
\begin{thm}\label{KernelHBM}
    Let $G$ be a Fuchsian group. The kernel of the differential of the higher Bers maps at the origin is given by
    $$ \mathrm{Ker}\:D_0\tilde \beta^\bullet_n \big|_{L_{(-1,1)}^\infty(\bb D^c, G)} = \psi_2^{-1} \left(\cal A^{n+1}_2(G)\right)^\perp\;,$$
    where $\perp$ denotes the orthogonal subspace for the pairing $\langle \cdot\:, \cdot\rangle^G_2$.
\end{thm}
The theorem, however, gains true content only after a more explicit description of the spaces    
$\cal A^l_q(G)$ and the inclusion $\cal A^l_q(G) \subset A_q^1(\bb D^c,G)$. \\

We first start with the case $G=\bb 1$. Reinterpreting Lemma 3 from \cite{An-sIEwAtQM} in this notation, it says that $\cal A^{2q}_q = A_q^1(\bb D^c)$. Observe that for $q=2$ this characterizes the kernel of the Bers embedding since then $2q = n+1$. In general, however, $\cal A^l_q$ will be a proper subspace of codimension $l-2q$, as the following proposition shows.
\begin{thm}
    For any  $l \geq 2q$ we have
    $$ A^1_q(\bb D^c) = \cal A^l_q \oplus \bb C[1/w]_{\mathrm{deg} < l-2q}\cdot w^{-2q}\;.$$
\end{thm}
\begin{proof} Let us again simplify things by the M\"obius transformation $\gamma:\bb D\rightarrow \bb D^c$ and denote the coordinate on the disc by $\eta$. $\gamma^*_q$ induces an isomorphism of the ambient space as usual; let's see what the elements of $\gamma_q^*\cal A^l_q$ look like:
$$ \tilde \omega^l_z(\eta) := \gamma_q^*\omega^l_z(w) = \frac{\eta^{-2q}}{(\frac{1}{\eta}-z)^l} = \frac{\eta^{l-2q}}{(1-\eta z)^l}\;.$$
Our strategy will be to compare power series expansions of functions. A sequence $\{f_j\} \in \cal O(\bb D)$ converges uniformly to $f$ iff the coefficients of their power series expansions converge to those of $f$, and uniform convergence of course implies $L^1$-convergence. Moreover, since we are working on the disc, for which $\lambda_{\bb D}^{2-q}$ is uniformly bounded (for $q \geq 2$), $L^q$-convergence implies $L^1_q$-convergence. Hence we are done if we can show that the coefficients of the power series converge to each other. The power series expansion of a function $\phi \in \gamma^*_q \cal A_q^l$ is given by
$$ \frac{\eta^{l-2q}}{(1-\eta z)^l} = \eta^{l-2q}\sum_{j=0}^\infty \frac{(-1)^{j}(l-1+j)!}{(l-1)!} z^j \eta^j:=\eta^{l-2q} \sum_{j=0}^\infty c_jz^j\eta^j \;,$$
so in particular the coefficients of order less than $l-2q$ are identically zero. Let us denote the truncated version of this series by 
$$ \tilde \omega^l_{z,N}(\eta) := \eta^{l-2q} \sum_{j=0}^{N-1} c_j z^j \eta^j \;.$$
Now let $f \in \tilde A^1_q(\bb D)$ be a function, $f = \sum b_{i+2q-l} \eta^i$ its power series, which converges uniformly on $\bb D$ since $f$ is holomorphic in some larger domain $D \supset \bar {\bb D}$, and split it as follows,
$$ f = f_0 + f_N + r_N\;,$$
where $f_0$ is the part of the expansion of order less than $l-2q$, $f_N$ the next $N$ terms and $r_N$ the remainder. By choosing $N$ points in $z_k(N) \in \bb D$ and $N$ numbers $b_k(N) \in \bb C$ appropriately we claim that we can achieve 
$$ 0 = \sum_{k=0}^{N-1} b_k(N) \tilde \omega^l_{z_k(N),N}(\eta) - f_N\;. $$
This amounts solving the following linear system for $b_k$,
\begin{eqnarray*}
    a_l = \sum_{k=0}^{N-1} c_k z_k^l b_k \;,\qquad l=0\ldots N-1\;,
\end{eqnarray*}
which is of course possible iff the determinant of the matrix
\begin{eqnarray*}
   \left( \begin{array}{cccc} c_0 & c_1& \ldots & c_{N-1}\\
                              c_0z_0 & c_1z_1 & \ldots & c_{N-1}z_{N-1}\\
                              c_0z_0^2 & c_1z^2_1 & \ldots & c_{N-1}z^2_{N-1}\\
                              \vdots & \vdots& \ddots&\vdots\\
                              c_0z_0^{N-1} & c_1z^{N-1}_1 & \ldots & c_{N-1}z^{N-1}_{N-1}\end{array} \right)         \end{eqnarray*}
does not  vanish. But the determinant is just a polynomial function on the product $\bb D^{N}$ which doesn't vanish identically, hence such choices are possible. Make such a choice for all $N$ and call the resulting functions $\phi_N:=\sum_{k=0}^{N-1} b_k \tilde \omega^l_{z_k,N}(\eta)$. The sequence $\phi_N$ obviously converges uniformly to $f-f_0$ and to an element of $\gamma^*_q\cal A^l_q$ since it is just the truncation of the power series of such a function. On the other hand, it is obvious that the term $f_0$ can not be approximated by elements of $\gamma^*_q\cal A^l_q$ because of the vanishing of the first coefficients. The proof concludes with the fact that any $g \in A_q^1(\bb D)$ can be approximated by elements of $\tilde A_q(\bb D)$ in the $A^1_q$-norm (see, e.g., Lemma 3 of \cite{An-sIEwAtQM}), so we have established
\begin{eqnarray}\label{DSD} A_q^1(\bb D) = \gamma^*_q\cal A^l_q \oplus \bb C[\eta]_{\mathrm{deg} P < l-2q}\;,\end{eqnarray}
which immediately implies the statement of the proposition.
\end{proof}
The $G$-version of this theorem is obtained as follows: Since $\Theta_q$ is a bounded operator it commutes with closure because of the closed graph theorem, so
$$ \mathrm{cl}_{A_q^1(\bb D^c, G)}\left \{ \Theta_q[\omega^l_z], z \in \bb D \right \} = \Theta_q \left[\mathrm{cl}_{A_q^1(\bb D^c)}\{\omega^l_z, z \in \bb D \}\right]\;,$$
which can also be written as
$$ \cal A_q^l(G) = \Theta_q[\cal A^l_q]\;.$$
Of course, $\Theta_q$ does not respect the direct sum decomposition \eqref{DSD}, 
so after applying $\Theta_q$ we merely get (observe that the action of $\gamma^*_q$ commutes with $\Theta_q$)
$$ A_q^1(\bb D, G) = \gamma^*_q \cal A^l_q(G) + \Theta_q\left(\bb C[\eta]_{\mathrm{deg} P < l-2q}\right)\;.$$
So in order to understand the mapping property of the differential in the G-setting we need to understand the intersection
$$ I_q^l(G):=\gamma^*_q \cal A^l_q(G) \cap \Theta_q\left(\bb C[\eta]_{\mathrm{deg} P < l-2q}\right)\;.$$
We will attack $I_q^l(G)$ with the help of a nice representation of the kernel of $\Theta_q$ given by Metzger. To determine the kernel of the Poincar\'e operator is a very old and very hard problem, especially if one considers the problem for general factors of automorphy. To the best of the author's knowledge, there is no general theorem or even  algorithm to determine whether a given function belongs to the kernel of $\Theta_{\rho_q}$ for an arbitrary factor of automorphy $\rho_q$. Nevertheless, the following theorem does yield information on the kernel of the Poincar\'e operator for the $q$-canonical factor and is enough for our purpose.
\begin{thm}[Metzger, \cite{TKotPSO}]
   Let $G$ be a Fuchsian group acting on $\bb D$, $q \geq 2$ an integer and $\Theta_q$ the Poincar\'e series operator of the q-canonical factor of automorphy for the group $G$. Let $p(z;k,g):=z^k - g(z)^k(g'(z))^q$. Then
   $$ \ker \Theta_q =\mathrm{cl}_{A_q^1(\bb D)}\left(\mathrm{span}_\bb C\left \{ p(z;k,g), \: k \in \bb N,\: g \in G\right\}\right)\;.$$
\end{thm}
We can now state:     
\begin{thm} \label{PSSpan}
   Let $G \neq \bb 1$ be Fuchsian. Then $\cal A_q^l(G) = A_q^1(\bb D^c, G)$. 
\end{thm}
\begin{proof} 
We need to write the functions $p(z;k,g):=z^k - g(z)^k(g'(z))^q$ as power series,
\benn
   p(z;k,g)&:=z^k - g(z)^k(g'(z))^q = z^k - \frac{(az+b)^k}{c^{k+2q}(z+d/c)^{k+2q}} \\&= z^k - \sum_{j=0}^{\infty} \frac{(-1)^j(k+2q+j)!}{(k+2q)!}\left(\frac{c}{d}\right)^{k+2q+j} z^j \\
   &= \tilde p(z;k,g) - R(z;k,g)\;,
\eenn
where $\tilde p(z;k,g)$ only contains terms up to the power $<l-2q$ in $z$ and $R(z;k,g)$ is the remainder. By linearity of $\Theta_q$ and the fact that $p(z;k,g)$ is in the kernel we get,
$$ \Theta_q[\tilde p(z;k,g)] = \Theta_q[R(z;k,g)]\;,$$
and evidently, $\tilde p(z;k,g) \in \bb C[z]_{<l-2q}$ and $R(z;k,g) \in \gamma^*_q\cal A^l_q$, so we know that
$$ \tilde P(G) := \mathrm{span}_{\bb C} \{\tilde p(z;k,g), \: k \in \bb N,\: g \in G\} \subset I_q^l(G)\;.$$
The remaining thing to understand is under which conditions $\tilde P(G) = \bb C[z]_{<l-2q}$. A first observation is that the kernel of the Poincar\'e series operator is infinite-dimensional if $G \neq \bb 1$, because of the different choices of $k \in \bb N$. Each choice of $k$ also affects the truncated series $\tilde p(z;k,g)$, so for any non-trivial group $G$ we already have infinitely many polynomials $\tilde p(z;k,g)$. We only have to make sure that this set contains a basis of $\bb C[z]_{<l-2q}$. But this is rather easy to see: First of all, for any non-trivial $g \in G$ thought of as an element of SL$(2,\bb C)$, the quantity $A:= |c/d| \neq 1$ has non-unit modulus. Then this implies that it is possible to choose $l-2q=:N$ natural numbers $k_n$ such that the equation
$$ \lambda_n \sum \tilde p(z;k_n,g) - z^d = 0$$
has a solution $\{\lambda_n(d)\} \in \bb C^N$ for all $d < N$, by writing this equation as a matrix equation $\lambda_n B_{nl}z^l =0$ with coefficient matrix obtained from the matrix
\begin{eqnarray*}
 \tilde B:= \left( \begin{array}{cccc} A^{k_1} & \ldots& (-1)^j\frac{(k_1 + 2q + j)!}{(k_1 + 2q)!}A^{k_1 + j} &\ldots\\
                           \vdots & \ddots &\vdots \\
                          A^{k_N} & \ldots& (-1)^j\frac{(k_N + 2q + j)!}{(k_N + 2q)!}A^{k_N + j} &\ldots \end{array}\right)
\end{eqnarray*}
by subtracting $1$ from the $d$-th column. But it is obvious by looking at the matrix that no two columns are linearly dependent even if $1$ should be subtracted from any of them as long as all $k_n$ are mutually different. Hence we have shown that $\tilde P(G) = \bb C[z]_{<l-2q}$, and therefore the proposition.
\end{proof}
\begin{cor} \label{InjOnTan}
  $D_0\beta^\bullet_n$ is injective when restricted to the tangent space of $\cal T_B(G)$ for any $G \neq \bb 1$.
\end{cor}
\begin{proof}
By the above considerations we know that for any $G \neq \bb 1$,
$$ \mathrm{Ker} D_0\tilde \beta^\bullet_n\big|_{L^\infty_{(-1,1)}(\bb D^c,G)} = \mathrm{Ker} D_0 \tilde \beta^\bullet_3\big|_{L^\infty_{(-1,1)}(\bb D^c,G)} = A^1_2(\bb D^c,G)^\perp\;.$$
\end{proof}
Hence, though the dimension of the kernel of the differential of $\beta^\bullet_n$ `grows linearly with $n$', the new null directions are all `strictly universal'.
\subsection{Higher Bers Maps: A Semi-Global Result}\label{HBM-SGR}
In this section we will prove results about the higher Bers maps that are of a different flavor, since they do not deal with the differential but with the map itself, and because of this, they also have to be proven separately for both series. The proof for the $A$ series a little bit more involved, since there we have to obtain information about the solution from information about the `inverse of the solution', whereas for the $B$ series this is more direct.\\

The results characterize all preimages of the origin of the maps $\beta_n^\bullet$, and hence are not local results, but on the other hand they are not global either, becuse we can only prove them for a single fiber of the mapping -  therefore we call them semi-global. They should be seen as the first step to proving the (conjectured) injectivity of higher Bers maps.\\

As already noted already, it is very hard to solve the higher Schwarzian differential equation explicitly for given $\phi \in B_{n-1}(\bb D,G)$. The situation simplifies a little bit for the homogeneous equation, i.e., $\sigma^\bullet_n[f] \equiv 0$, which we will study first. 

Recall Lemma \ref{AltDefHS}, which gave us the following representation of the higher Schwarzian derivative,
$$ \sigma^A_n[f] = \left(\frac{d^{n-3}}{dz^{n-3}} S[f^{-1}] \right)\circ f \cdot (f')^{n-1}\;.$$ 
Hence $\sigma^A_n[f] \equiv 0$ implies 
$$ 0 \equiv \left(\frac{d^{n-3}}{dz^{n-3}} S[f^{-1}] \right)\circ f\;.$$
Now a holomorphic function whose $(n-3)^{\mathrm{th}}$ derivative vanishes identically must be a polynomial of degree $\leq n-4$, hence:
\begin{lemma} \label{HomSol}
	$f$ is a solution to $\sigma^A_n[f] \equiv 0$ iff the Schwarzian of its inverse function is a polynomial of degree 			at most $n-4$.
\end{lemma}
This implies, in contrast to the ordinary Schwarzian, which is zero only for M\"obius transformations, that the higher Schwarzians kill many schlicht functions, and indeed many elements of $\cal T_S(\bb 1)$. We state this in the form of a lemma.
\begin{lemma}
   For any integer $n\geq 4$, there exist non-trivial elements $f \in \cal T_S(\bb 1)$ such that $\sigma^A_n(f) \equiv 0$.
\end{lemma}
\begin{proof} Fix $n$, then pick a polynomial $P$ of degree $N\leq n-4$ and a quasidisc $D\subset \bb C$ such that
$$ \| P \|_{B_{2}(D)} < \delta(D)\;,$$
where $\delta(D)$ is the constant from Theorem \ref{SchlichtCrit}, i.e., the constant such that any function with $B_2$ norm less than $\delta(D)$ is schlicht in $D$. This is of course possibe, e.g., choose $D=\bb D$, 
\begin{eqnarray*}
\|P\|_{B_2(\bb D)} &=& \mathrm{sup}_{z \in \bb D} |P(z)| \lambda_{\bb D}^{-2} (z) \leq \mathrm{sup}_{z \in \bb D}|P(z)| \cdot \mathrm{sup}_{z \in \bb D}\lambda_{\bb D}^{-2} (z) \\
&\leq & \mathrm{sup}_{z \in \partial \bb D}\big|\sum_{i=0}^{N} a_i z^i\big| \leq \sum_{i=0}^{N} |a_i|\;,
\end{eqnarray*}
so any polynomial for which the sum of the moduli of the coefficients is less than $\delta(\bb D) = 2$ will do. Pick a function $g \in \cal M_{\mathrm{li}}(\bb D)$ such that $S_g = P$. Then $g$ is schlicht and moreover has a quasiconformal extension to the Riemann sphere\footnote{This follows from a modified version of the Ahlfors-Weil section, which also exists for any quasidisc $D$.} by Theorem \ref{SchlichtCrit}. Finally, pick a M\"obius transformation $M$ such that for $g_M:= M \circ g$, we have $g_M(\bb D) \supset \bb D$. By construction, $g_M^{-1}|_{\bb D} \in \cal T_S(\bb D)$ and $S_{(g_M^{-1})^{-1}} = S_{g_M} = P$, hence $\sigma^A_n(g_M^{-1}) \equiv 0$. 
\end{proof}
Surprisingly, the $\sigma^A_n$ and hence the higher Bers maps have an injectivity property when they are restricted to any finite-dimensional Teichm\"uller space. Theorem \ref{HighInj} below makes this precise. But befoer we come to the theorem we will state two other theorems which will be needed in the proof. \\ \\
First, we remind the reader of the classical theorem on the uniform convergence of normalized Riemann mappings for converging domains.
\begin{thm}\label{RMTconv}
   Let $\{D_j\}$ be a sequence of sc-hyp domains containing a common point $z_0$ converging to $D_0$ in the sense that 
   $$\overline{ D}_{j+1} \subset D_j\qquad \mathrm{and} \qquad \mathrm{int}\: \Big(\bigcap_j D_j\Big) = D_0\;,$$
   and let $\psi_j:\bb D \rightarrow D_j$ be the Riemann mappings normalized by $\psi_j(0) = z_0$, $\psi'(0)>0$. Then 		 the $\psi_j$ converge uniformly to a normalized Riemann mapping $\psi_0: \bb D\rightarrow D_0$.
\end{thm}
The second and crucial theorem we need concerns the geometry of the limit set of a quasi-Fuchsian group.
\begin{thm} [\cite{UFaTS}, Thm.~ IV.4.2] \label{LimitSet}
   Let $G'$ be a quasi-Fuchsian group of first kind. Then $\Lambda(G')$ is either a circle on $\cinf$ or it is not         differentiable on a dense subset.
\end{thm}
Finally, having these we can state and prove the following.
\begin{thm} \label{HighInj}
   Let $G$ be Fuchsian of first kind. Then $\left(\sigma^A_n\right)^{-1}(0)\cap \cal T_S(G) = \{\bb 1_{\bb D}\}$.
\end{thm}
\begin{proof} First of all, it is clear that the function $w \mapsto w$ is mapped to the origin in $B_{n-1}(\bb D)$ since it is a M\"obius transformation. So for the rest of the proof we assume $f \in \cal T_S(G)$ and $\sigma^A_n[f] \equiv 0$. According to Lemma \ref{HomSol}, 
$$\sigma^A_n[f]\equiv 0 \qquad \Longleftrightarrow \qquad S_{f^{-1}} \in \bb C[w]_{\leq n-4}\;.$$
In particular $S_{f^{-1}}$ is an entire function and so it extends to any domain $D \supset f(\bb D)$. Hence we know that $f^{-1}$ extends to a meromorphic locally injective function on any such domain $D$. The rest of the proof consists of finding a simply-connected domain $D \supset f(\bb D)$ such that the extension is holomorphic and schlicht. For suppose we have such a domain; then by taking the inverse we have a schlicht extension of $f$ to a domain containing the unit disc. In particular, this domain contains an open arc of $\partial \bb D$ and so the corresponding part of $\partial f(\bb D)$ is a holomorphically embedded curve. But this contradicts the assumption that $f \in \cal T_S(G)$ since limit sets of quasi-Fuchsian groups of first kind are known to be either circles or nowhere rectifiable curves (see Theorem \ref{LimitSet}).\\ \\
In case $f(\bb D)$ is starlike with respect to some $p \in f(\bb D)$, the argument is direct. Without loss of generality assume $p=0$, else we can acheive this by composing back and forth by translations. Then define
$$D_\delta := \{ (1+\delta)z\;, z \in f(\bb D)\}\;,$$
which of course is nothing else than the image of $f(\bb D)$ under the M\"obius transformation $\cal E_\delta(z) = (1+\delta)z$ and is a simply-connected domain containing $f(\bb D)$. Let $f_\delta^{-1}$ denote the extension of $f^{-1}$ to $D_\delta$. We pull back the extension to $f(\bb D)$ and observe that 
\begin{eqnarray}\label{schlicht}
 \cal E_\delta^* f_\delta^{-1} \in \cal S(\bb D) \qquad \Leftrightarrow \qquad f_\delta^{-1} \in \cal S(\bb D_\delta)\;.
\end{eqnarray}
But the first statement can be verified by means of the norm of the Schwarzian derivative. We know that $\bb S$ is closed in $B_2(f(\bb D))$ and that $S_{f^{-1}}$ is an interior point since $f$ was assumed to be in $\cal T_S(G) \subset \cal T_S(\bb 1)$ and $T(\bb 1) = \mathrm{int}\: \bb S$. Now
\begin{eqnarray*}
	S_{\cal E_\delta^* f_\delta^{-1}} &=& S_{f_\delta^{-1} \circ \cal E_\delta} = \left(S_{f_\delta^{-1}} \circ \cal 	E_\delta\right)( \cal E_\delta')^2 + S_{\cal E_\delta}\\
	&=& \left(S_{f_\delta^{-1}} \circ \cal 	E_\delta\right) (1+\delta)^2\;.
\end{eqnarray*}
Hence
\benn
 \|S_{E_\delta^* f_\delta^{-1}}  - S_{f^{-1}}\|_{B_2(f(\bb D))} &= \sup_{z \in f(\bb D)}\left( |P(1+\delta z)(1+\delta)^2 - P(z)|\lambda^{-2}_{f(\bb D)}(z)\right)\\
  &\leq C \cdot \sup_{z \in f(\bb D)} |P((1+\delta) z) - P(z)| + O(\delta)\;,
\eenn
which can be made smaller than any given $\epsilon$ since a polynomial is of course uniformly continuous. This means that for small enough $\delta$, $E_\delta^* f_\delta^{-1}$ is also an interior point of $\cal S(f(\bb D))$, hence schlicht. By (\ref{schlicht}), $f^{-1}_\delta$ is then a schlicht extension of $f^{-1}$.\\

In general, $f(\bb D)$ is a Jordan domain for $f \in \cal T_S(\bb 1)$. We now choose a sequence of domains $\{D_j\}$ approximating $f(\bb D)$ in the sense that
$$ \overline{D}_{j+1} \subset D_j \qquad \textrm{and} \qquad \bigcap_j D_j = f(\bb D)\;,$$ and denote by $\psi_j:\bb D \rightarrow D_j$ the Riemann mapping fixing the origin with $\psi'(0) > 0$. By Theorem \ref{RMTconv}, the sequence $\{\psi_j\}$ converges uniformly to the normalized Riemann mapping of $f(\bb D)$, which by our normalization is $f$ itself. Denote by $f^{-1}_j$ the extension of $f^{-1}$ to $D_j$. We now obtain a sequence of functions $\{(\psi_j \circ f^{-1})^*f_j^{-1}\}$ on $f(\bb D)$; we claim that they are schlicht for big enough $j$. We show this in the same way as before, namely by estimating the differences of $B_2$-norms,
\benn
   &\| S_{(\psi_j \circ f^{-1})^*f_j^{-1}} - S_{f^{-1}}\|_{B_2(f(\bb D))}= \\ &\|S_{f^{-1}} \circ (\psi_j \circ f^{-1})((\psi_j \circ f^{-1})')^2 + S_{\psi_j \circ f^{-1}} - S_{f^{-1}}\|_{B_2(f(\bb D))}\;.
   \eenn
Now, by construction, $(\psi_j \circ f^{-1})$ converges uniformly to the identity function, hence its Schwarzian derivative converges uniformly to zero, and so
\begin{eqnarray*}
	\|S_{(\psi_j \circ f^{-1})^*f_j^{-1}} - S_{f^{-1}}\|\leq C \cdot \sup_{z \in f(\bb D)} |P((\psi_j \circ f^{-1})(z)) - P(z)| + O(j)\;.
\end{eqnarray*}
$O(j)$ is, by slight abuse of notation, an expression which goes to zero as $j$ goes to $\infty$. Therefore the same conclusion as in the starlike case is valid.
\end{proof}
Let us now look at the $B$ series. Here one can actually write down the explicit solution to the equation $\sigma^B_n[f]\equiv 0$. In fact,
$$ 0 = \sigma^B_n[f] = -2(f')^{\frac n2 - 1}\frac{d^{n-1}}{dz^{n-1}}\left( (f')^{1-\frac n2}\right)\;,$$
immediately implies
\begin{equation} \label{SoltoBSeries} f' = \left(\alpha_0 + \ldots + \alpha_{n-2}z^{n-2}\right)^{-2\frac 1{n-2}}\;,\end{equation}
for arbitrary coefficients $\alpha_i \in \bb C$. This equation can of course be integrated easily, and hence we see that the operators of the $B$ series also always have non-trivial homogeneous solutions for $n\geq 4$. Now, by some minor modifications, the second half of the proof of Thm.~ \ref{HighInj} also yields a proof of the analogous theorem for the $B$ series, because all that was needed for the contradiction was that there is some arc on $\partial \bb D$ over which the function extends holomorphically. Obviously, the solutions given in \eqref{SoltoBSeries} satisfy this for almost all of $\partial \bb D$, except for possible isolated zeroes of the polynomial $\alpha_0 + \ldots + \alpha_{n-1}z^{n-2}$.  
\begin{thm}\label{HighInj2}
   Let $G$ be Fuchsian of first kind. Then $\left(\sigma^B_n\right)^{-1}(0)\cap \cal T_S(G) = \{\bb 1_{\bb D}\}$.
\end{thm}
Of course both of the theorems in this section can be extended to such Fuchsian groups of second kind, for which similar facts about the fractal nature of their limit sets is known.
\subsection{Higher Bers Maps: Further Remarks, Questions and Future Work}\label{HBM-future}
In this section we remark on the implications of the previous results and comment on difficulties as well as pose questions which naturally come to mind and constitute potential future research.

\subsubsection{Dependence on the base point and chain rule of higher Schwarzians}
By this we of course mean that we only succeeded in studying the differential of the higher Bers maps at the origin. As mentioned earlier, this is precisely the way Bers proceeded in his study of the Bers embedding, only that in that case he had the chain rule
$$ S_{f \circ g} = (S_f \circ g) \cdot (g')^2 + S_g $$
at his disposal, which after taking the derivative of this equation produced a relation between the differentials at different points (similarly to Equation \eqref{Isos}). Of course, by iterative use of the chain rule for functions, a similar thing can \emph{in principle} be written down for the $\sigma^\bullet_n$ for \emph{fixed n}, but the shape of the equation will depend heavily on $n$. Some closed form of the chain rule would be nice, although highly unlikely to exist because of the \emph{non-homogenity} of $\sigma^\bullet_n$ (see Section \ref{OtherWork}). 
\begin{ex}
   Is there a good formulation of a chain rule for higher Schwarzians? 
\end{ex}
Philosophically thinking, however, the maps should only quantitatively depend on the base point, not qualitatively, since it does not matter how we uniformize the Riemann surface in the beginning, so there is strong reason to believe that $\beta^\bullet_n$ is an embedding of $\cal T_S(G)$ into $B_{n-1}(\bb D)$. Also the semi-global result (Thm.~ \ref{HighInj}) gives further reason to hope that this is the case. On the other hand, observe that any complex linear combination $\alpha \sigma_n^A + \beta \sigma_n^B$ also induces a holomorphic map (which also should be called a higher Bers map), and certainly injectivity is not a phenomenon that is necessarily well-behaved under addition.
\begin{ex}
   Are (some of) the higher Bers maps $\beta^\bullet_n$ embeddings? 
\end{ex}
\subsubsection{Solution Theory of Higher Bers Maps}
The second major difference in the theory for $n>3$ as opposed to $n=3$ is the solution theory of the equation $\sigma^\bullet_n[f] = \phi$. For $n=3$, although this is still a complicated, non-linear, 3rd-order differential equation, we have full control over the solutions as described in Theorem \ref{SchwSol} due to the fact that there is an associated \emph{linear} ODE. This solution is a great tool; e.g., this is needed in the explicit construction of the Ahlfors-Weill section, which again has great theoretical impact on Teichm\"uller theory.\\

We want to point out briefly the difficulties of a straight-forward generalization of the solution scheme to higher Schwarzians: First of all, let there be a linear equation of order $k$ with holomorphic non-vanishing leading coefficient in analogy to the linear ODE for the Schwarzian. Such an equation would then have $k$ independent solutions $\phi_i$, out of which we would have to construct a locally injective function $f$ in such a way that a change of basis $\tilde \phi_i = \Lambda_i^j \phi_j$ produces a new solution $\tilde f$ with same value under the higher Schwarzian, i.e., $\sigma^\bullet_n[\tilde f] = \sigma^\bullet_n[f]$. This alone, i.e., to describe the fibres of the operators $\sigma^\bullet_n$ is a very difficult.
\begin{ex}
   Is there a nice, explicit and useful solution theory for the higher Schwarzians which gives more insight on the geometry of the higher Bers maps? 
\end{ex}
As a side note: Although not satisfying the last requirement, Kim does something very interesting in \cite{SDaM}. He treats the general holomorphic ODE of degree $d$ with non-vanishing leading coefficient, 
$$ y^{(d)} + p_{d-2}y^{(d-2)} + \ldots + p_0 y = 0\;,$$
and defines $f_i=\phi_i / \phi_{d}$, for $i=1,\ldots , d-1$, where the $\phi_i$ are the linearly independent solutions. He is then able to find general expressions for the coefficients 
$$p_l = \Phi_l(f_1,\ldots, f_{d-1})$$
in terms of the quantities $f_i$ (\cite{SDaM}, Thm.~ 2.1). These are quite involved so we won't reproduce them here. For $d = 2$, $\Phi_0$ is precisely the Schwarzian, of course. For $d=3$, the expression for $\Phi_0$ and $\Phi_1$ are also given explicitly in the paper (p.~ 4, middle). A quick glance at them, however,  reveals that they have nothing to do with the higher Schwarzian derivatives, because they have \emph{second order derivatives of the $f_i$ in the denominator}.
\subsubsection{The space $B_\infty(\bb D)$ and out- resp.~ inradii of $T^\bullet_n(G)$} 
Recall that for a bounded domain $D$, $B_q(D) \subset B_{q'}(D)$ for $q \leq q'$, so let us denote the inclusion by $i_q^{q'}$. 
This data defines an inductive system. We will denote the corresponding inductive limit space by
$$ B_\infty(\bb D) := \bigcup_{j=2}^{\infty} B_j (\bb D)\;.$$
This space naturally carries the \emph{final topology}, sometimes also called the \emph{inductive topology}, which is by definition the finest topology such that all the induced maps $i^\infty_q: B_q(\bb D) \rightarrow B_\infty(\bb D)$ are continuous. This inductive system is, however, not \emph{strict} in the sense of (\cite{ACiFA}, Def. IV.5.12), since the topology of $B_q$ induced by the norm $\|\cdot\|_{B_{q'}}$, i.e., the induced topology on $B_q$ viewed as a subspace of $B_{q'}$, does not agree with the norm topology of $B_q$, and under this circumstance it is hard to say something about the limit space in general, even when the sequence consists of Banach spaces.\\ \\
As a first observation, the norms of the inclusions are uniformly bounded.
\begin{lemma}
   For any $f\in B_N(\bb D)$, the sequence $\|i^{n}_N f\|_{B_n(\bb D)}$ converges to $f(0)$ for $n\rightarrow \infty$.
\end{lemma}
\begin{proof} If $f$ is constant the statement is trivial. Else as a nonconstant holomorphic function, $|f|$ approaches its supremum as $z \rightarrow \partial \bb D$. Since $f$ is in $B_N(\bb D)$, $|f|\lambda^{-N}:= M$ is bounded, and so since $1/\lambda_{\bb D} < 1$ except at the origin,
$$ |f(z)| \lambda_{\bb D}^{-N-j}(z) \rightarrow 0 \qquad \mathrm{for} \;j\rightarrow \infty \; \mathrm{and} \; z \in \bb D \backslash \{0\}\;.$$ 
On the other hand, $\lambda_{\bb D} (0) = 1$ so there the above sequence equals $|f(0)|$.
\end{proof}
So far nothing really interesting has happened. It becomes more interesting if we look at a more interesting sequence in $B_\infty(\bb D)$ induced by the higher Schwarzians. Any $f \in \cal S(\bb D)$ 
induces two sequences
$$s^\bullet_j(f):=\sigma^\bullet_j[f] \subset B_{j-1}(\bb D)\;,$$
which can of course also be viewed as a sequence $\{S^\bullet_j:= i^\infty_j s^\bullet_j(f)\}$ in $B_\infty(\bb D)$.
\begin{ex}
   For which $f$ does such a sequence converge resp.~ diverge? If $f \in \cal T_S(G)$, does the behaviour depend on the     group $G$? Can one reasonaby define $T^\bullet_\infty(G) \subset B_\infty(\bb D)$ for some Fuchsian groups?
\end{ex}
Recall that although the images $\bb S_n$ in $B_{n}(\bb D)$ are bounded, the sharp bound is given by $c^A_n := 4^{n-3}(n-2)!6$ and \mbox{$c^B_n := 2(n-2)n \ldots (3n-6)$} by Proposition \ref{higherbound} resp.~ \ref{higherbound2}, which grow (very) quickly with $n$. A related question is the following.
\begin{ex}
   How do the outradii $o^\bullet_n(G)$ of $T^\bullet_n(G)$,
   $$ o^\bullet_n(G) := \mathrm{sup}_{f \in \cal T_S(G)} \| s^\bullet_n(f)\|_{B_{n-1}(\bb D)}\;,$$
   behave, and how do they depend on the nature of the group $G$?
\end{ex}
Of course $o^\bullet_n(\bb 1) = c^\bullet_n$, but recall for the Bers embedding there exist groups $G$ for which $o_3(G)<o_3(\bb 1)$, e.g., finitely generated Fuchsian groups. One could expect that some similar behaviour is reflected in the higher outradii.\\

Similarily one can define the $n$-th inradius of $T^\bullet_n(G)$ as
$$ i^\bullet_n(G) := \sup_{\delta \in \bb R} \{\bb D_\delta \subset T^\bullet_n(G)\} \;.$$
Note that $i^\bullet_n(G) = 0$ for groups of first kind by dimensional reasons and that $i^\bullet_n(\bb 1) > 0$ by the implicit function theorem. An estimate for this number would yield as a corollary a new criterion for univalence resp.~ quasiconformal extendability of a function $f \in \cal O(\bb D)$.\\ \\
One approach to this question is to study the $n$-th Ahlfors-Weill sets $AW_n^\bullet \subset B_{n-1}(\bb D)$, which are defined to be the images
$$ AW_n^\bullet := (\beta_n^\bullet \circ s)\left(\bb B_2(B_2(\bb D))\right)$$
of the ball of radius two under the Ahlfors-Weill section composed with an $n$-th higher Bers map. As mentioned previously, the $AW_n^\bullet$ have non-empty interior.
\begin{ex}
    Are the Ahlfors-Weill sets domains?
\end{ex}
\subsubsection{Other Operators}\label{OtherWork}
As we showed in Theorem \ref{higherBersholomorphic}, there are a wealth of operators that lead to holomorphic mappings of Teichm\"uller space into the Banach spaces $B_m(\bb D)$. More precisely, any differential operator $Q$ which maps schlicht functions to holomorphic functions, for which the expression $Q[f]$ is a polynomial over $\bb C$ in $f'', \ldots, f^{(N)}$ and $(f')^{-1}$ and which satisfies 
\begin{equation}\label{trafoBehImp}Q[f\circ g] = (Q[f] \circ g)(g')^m\;,\end{equation}
induces a holomorphic map $\beta^Q:\cal T_B(G) \rightarrow B_m(\bb D)$.\\

In Theorem \ref{ThmDifferential} we proved that a large subclass of these operators have the same differential at the origin and this differential is surjective by Theorem \ref{HBDiffSurj}. Any estimate on the inradius of $i^Q(\bb 1)$ for such an operator $Q$ yields a criterion for quasiconformal extendability, and in particular, for schlichtness of functions.
\begin{ex}
   Can one obtain a general theorem on the universal inradii $i^Q(\bb 1)$ of the class of operators satisfying the prerequisites of Theorem \ref{ThmDifferential} in terms of the coefficients of the operators? Does this yield a set of new and systematic criteria for schlichtness?
\end{ex}
The paper \cite{MIOoRS} should be consulted here, since it contains a rather complete classification of operators satisfying \eqref{trafoBehImp}.\\

There is one class of operators we want to mention here that has been studied previously by Harmelin in \cite{UFXXX}, which he called \emph{homogeneous operators} or \emph{homogeneous higher Schwarzians}. They don't satisfy the crucial requirement \eqref{trafoBehImp}, which entered many proofs in the present paper, yet still lead to holomorphic mappings. The holomorphicity is estabilshed with the help of the homogenity of the operators. Let us give a more precise definition: A differential operator is called a \emph{homogeneous higher Schwarzian}, iff it is a polynomial,
$$ P_N[f] = \sum a_I \phi_I(f)\;, \quad \phi_I(f) = \phi_{i_1}(f)\ldots\phi_{i_{j(I)}}(f)\;, $$
where $\phi_n(f) := (S_f)^{(n-2)}$ and $I$ is a multiindex where each index has values in $\{0,2,3,\ldots,N\}$ and each of the monomials $\phi_I$ has the same total number of derivatives of the Schwarzian, i.e., $|I|=\sum n_i=N$ for all $I$ in the sum. Since the derivative of an element of $B_n(\bb D)$ is an element of $B_{n+1}(\bb D)$, $P_N(f)$ is a bounded $N$-differential. And because of this homogenity in the Schwarzian, these operators are M\"obius invariant with respect to postcomposition, i.e., $P[g\circ f] = P[f]$.\\

For to be applicable in Teichm\"uller theory, however, it is clear that the operators have to incorporate the group $G$ in some meaningful way, as for example equation \eqref{trafoBehImp}, and hence the homogeneous higher Schwarzians are not interesting from the point of view of Teichm\"uller theory. The only point these operators have in commom with the class of operators described in Theorem \ref{ThmDifferential} is the differential at the origin.

\end{document}